\documentclass[11pt]{article}\input{amssym.def}\input{amssym}
\usepackage[active]{srcltx}
\usepackage{color}
\usepackage{tikz}
\usepackage{relsize}

\usepackage{hyperref}
\hypersetup{
    colorlinks=true, 
    linktoc=all,     
    linkcolor=blue,  
}

\newenvironment{annotacia}{\centerline{\sc Abstract}\vspace{2mm}\narrower\narrower\sf}

\def\theequation{\thesection.\arabic{equation}}
\makeatletter\@addtoreset{equation}{section}\makeatother

\def\nn{\nonumber}\def\lb{\label}
\def\be{\begin{equation}}
\def\ee{\end{equation}}
\def\ba{\begin{eqnarray}}
\def\ea{\end{eqnarray}}
\def\tr{{\rm Tr}\,}
\def\Tr#1{{\rm Tr}_{\! R^{\mbox{\scriptsize$(#1)$}}}}

\newcounter{theorem}\makeatletter
\@addtoreset{theorem}{section}\makeatother

\newtheorem{prop}[theorem]{Proposition}\newtheorem{rem}[theorem]{Remark}\newtheorem{lem}[theorem]{Lemma}
\newtheorem{def-lem}[theorem]{Definition-Lemma}\newtheorem{def-prop}[theorem]{Definition-Proposition}
\newtheorem{theor}[theorem]{Theorem}
\newtheorem{cor}[theorem]{Corollary}

\begin{document}
\title{ }
\rightline{To our friend Andrey Pogrebkov~~}

\rightline{on the occasion of his 80th birthday}
\vspace{1cm}

\begin{center}
	{\Large \textbf{Characteristic identities for orthogonal\\[3pt]
			quantum matrices: examples}}
	
	\vspace{1cm} {\large \textbf{Oleg Ogievetsky$^{\diamond\,\dag}$ 
			and Pavel Pyatov$^{\ast\,\star
			}$}}
	
	\vskip .8cm $^{\diamond}$Aix Marseille Universit\'{e}, Universit\'{e} de
	Toulon, CNRS, \\ CPT UMR 7332, 13288, Marseille, France
	
	\vskip .3cm $^{\dag}${I.E.Tamm Department of Theoretical Physics, 
		P.N. Lebedev Physical Institute, Leninsky prospekt 53, 119991 Moscow, Russia}
	
	\vskip .3cm $^{\ast}${National Research University "Higher School 
		of Economics", 20 Myasnitskaya street, Moscow
		101000, Russia}
	
	\vskip .3cm $^{\star}${Bogoliubov Laboratory of Theoretical Physics, 
		JINR, 141980 Dubna, Moscow region, Russia}
\end{center}

\begin{annotacia}\noindent 
	In this note, we illustrate our recent results regarding structure theory of  orthogonal Quantum Matrix algebras  \cite{OP-BMW,OP-reciprocal,OP-ortho} using the simplest nontrivial example: the RTT-algebra  associated with the canonical orthogonal R-matrix of type $O(3)$ (see, e.g., \cite{FRT}). By introducing a suitable linear order on the set of generators, we obtain a set of quadratic relations defining the algebra. Next, we investigate relations among its higher-degree polynomials. As a result, we verify `in components' the Cayley-Hamilton and the reciprocal relations, obtained earlier with the use of the R-matrix techniques. We explicitly calculate the coefficients of the characteristic polynomials of quantum matrices and  propose a
	conjecture regarding the form of the Hilbert-Poincar\'e series for the RTT-algebra.

\end{annotacia}

\newpage
\tableofcontents
\bigskip\bigskip\bigskip

\section{Introduction}    

\def\theequation{\thesection.\arabic{equation}}
\makeatletter\@addtoreset{equation}{subsection}\@addtoreset{equation}{section}\makeatother

The story of the noncommutative versions of the fundamental theorem of matrix algebra --- the Cayley-Hamilton theorem --- dates back to the early 1970s with the observation that the generators of the universal enveloping algebra 
$U(g)$ arranged in matrices, satisfy a certain version of the characteristic identity.
These results, first obtained at the level of representations \cite{BGr,Gr}, subsequently received an abstract-algebraic interpretation \cite{BCC} and were generalized to the case of Lie superalgebras (for a survey and list of references, see \cite{IWG}).

With the discovery of quantum groups \cite{D,Jimb,FRT}, the list of associative algebras generated by mutually non-commuting matrix components has expanded significantly.
Examples of characteristic identities have been derived for quantized universal enveloping algebras \cite{GZB,Md}, for their dual algebras of quantized functions on Lie groups (so-called RTT-algebras) \cite{EOW,Zh,OV}, and for Reflection Equation (RE-) algebras \cite{NT,PS}\footnote{This family of algebras was introduced into consideration in \cite{Ch,KS}.}.

In \cite{GPS1},  the Cayley–Hamilton theorem was proved in a unified manner for a family of RE-algebras defined by $GL(n)$-type R-matrices. It turned out that the R-matrix techniques employed in \cite{GPS1} also yield structural results for RTT-algebras \cite{PS2,IOPS,IOP}. Subsequently, the Cayley-Hamilton theorem was generalized to a broader class of $GL(n)$ type Quantum Matrix (QM-) algebras that encompasses both  RE- and  RTT-algebras \cite{IOP1,IOP2}. To date, this line of research has also led to the construction of characteristic identities for families of QM-algebras of type  $GL(m|n)$ \cite{GPS2,GPS3}\footnote{See \cite{KT1,KT2} for results in the classical supermatrix case.} and for QM-algebras of symplectic $Sp(2n)$ and orthogonal $O(n)$ types \cite{OP-BMW,OP-SpCH,OP-reciprocal,OP-ortho}.

Many other studies of characteristic identities and related issues in the context of quantum group theory have been limited to the canonical (Drinfeld-Jimbo's) quantization of universal enveloping algebras and their dual RTT-algebras (see, e.g., \cite{GZB,NT,Zh,OV,MRS,GLZ,Md,JW}). 
In these cases, known explicit relations for the algebra's generators can be used to construct a linearly ordered  basis for the space of polynomials in these generators, which can then be applied in further reasoning and calculations.
This approach is well-suited for the case of general linear quantum (super)groups. However, in orthogonal quantum matrix algebras, the set of relations turns out to be redundant: in addition to ordering relations, it includes supplementary orthogonality conditions  (see \cite{FRT}). In such a situation, constructing a linear basis for a quantum matrix algebra and using it for calculations becomes problematic.

In this note, we consider the simplest non-trivial example of the orthogonal $O(3)$ type RTT-algebra 
${\cal F}(R)_{q,v}$. The indices in the notation  are explained by the 
definition of this algebra which employs a `twisted' 
two-parameter generalization of the standard orthogonal R-matrix acting on the 
tensor square of the three-dimensional space. The R-matrix parameters are denoted by $q$ and $v$; $q$ is responsible 
for the `quantization', while $v$  arises when twisting the standard R-matrix\,\footnote{For a discussion of twist transformations of the R-matrix, see, e.g., \cite{Resh,IOP3}.}. 

We establish a suitable linear order on the set of the algebra generators and explicitly write down the ordering and orthogonality relations for these generators. We then investigate the impact of the orthogonality conditions on the construction of a linear basis in ${\cal F}(R)_{q,v}$ and propose a formula for its Hilbert-Poincar\'e series.   Using this linear basis, we re-verify -- for the case of ${\cal F}(R)_{q,v}$ -- the results previously obtained for the general case of the orthogonal QM-algebra using R-matrix techniques \cite{OP-reciprocal,OP-ortho}. Finally, we consider ${\cal F}(R)_{q,v}$ at the particular value of the twist parameter $v=1$, at which two important elements of this algebra -- $g$ and ${\goth e}_3$ (see Section \ref{RTT-alg}) -- called, respectively, the contraction and the 3-elementary sum become central. A quotient  algebra ${\cal F}_q(O(3)):={\cal F}_{q,v=1}/\langle q^{2(n-1)}g=1 \rangle$ can be viewed as a quantized algebra of functions on the orthogonal Lie group $O(3)$. We construct a linear basis for this algebra and present its Hilbert-Poincaré series. 
\smallskip

Summarizing our study of the algebra ${\cal F}(R)_{q,v}$, we arrive at the following conclusions.  Naturally, the use of a linear basis yields explicit formulas for important elements of the QM-algebra --- the coefficients of the characteristic identity. However, as our simple example demonstrates, proving structural results using a linear basis turns out to be extremely laborious, not very clarifying, and unlikely to be feasible for a broad class of QM-algebras.

\section{Multiparametric orthogonal R-matrices in dimension 3}
In this preparatory Section, we have gathered the necessary information about a specific orthogonal $O(3)$ type $R$-matrix, which we will use later for the construction and study of the RTT-algebra. For the sake of completeness, we present expressions for multiparameter orthogonal $R$-matrices of type $O(n)$, $n$ arbitrary, in the Appendix.
\medskip

The standard orthogonal R-matrix $R^\circ\in {\rm Aut}(V^{\otimes 2})$, acting on the tensor square of the three-dimensional space $V$, has the form \cite{FRT}:
\[
R^\circ=\left(\begin{array}{ccc|ccc|ccc}q&.&.&.&.&.&.&.&.\\.&\l&.&1&.&.&.&.&.\\
.&.&\l(1-\frac{1}{q})&.&-\frac{\l}{\sqrt{q}}&.&\frac{1}{q}&.&.\\
\hline
.&1&.&.&.&.&.&.&.\\
.&.&-\frac{\l}{\sqrt{q}}&.&1&.&.&.&.\\
.&.&.&.&.&\l&.&1&.\\
\hline
.&.&\frac{1}{q}&.&.&.&.&.&.\\
.&.&.&.&.&1&.&.&.\\
.&.&.&.&.&.&.&.&q
\end{array}\right), \quad \lambda:=q-q^{-1}.
\]
Here, dots mark the positions of the zero components of the R-matrix.
In this note, we will use a `twisted' version of this R-matrix equipped with two additional non-zero complex parameters $u$ and $v$
(2 is the maximal number of diagonal twisting parameters for the matrix $R^\circ$):
\be
\lb{R-O3}
R=\left(\begin{array}{ccc|ccc|ccc}q&.&.&.&.&.&.&.&.\\.&\l&.&v&.&.&.&.&.\\
.&.&\l(1-\frac{1}{q})&.&-\frac{\l u}{\sqrt{q}}&.&\frac{v^2}{q}&.&.\\
\hline
.&\frac{1}{v}&.&.&.&.&.&.&.\\
.&.&-\frac{\l}{u\sqrt{q}}&.&1&.&.&.&.\\
.&.&.&.&.&\l&.&v&.\\
\hline
.&.&\frac{1}{qv^2}&.&.&.&.&.&.\\
.&.&.&.&.&\frac{1}{v}&.&.&.\\
.&.&.&.&.&.&.&.&q
\end{array}\right). \phantom{\quad \lambda:=q-q^{-1}}
\ee
The parameter $u$ can be assigned an arbitrary non-zero value by means of a diagonal change of basis in the space $V$; we will choose it in such a way as to  avoid the appearance of square roots in the components of the R-matrix:
\be
\lb{fix-u}
u=\sqrt{q}\ .
\ee

On the contrary, the parameter $v$ is essential in the sense that it cannot be altered by a change of basis; 
it is related to the simple R-matrix twist transformation (see \cite{Resh,IOP3}).
Unlike the purely gauge parameter $u$, the parameter $v$ significantly affects the RTT-algebra associated with the R-matrix (\ref{R-O3}). We shall keep this parameter arbitrary until the last Section of the paper.
\smallskip
 
The R-matrix (\ref{R-O3})  satisfies the braid relation, also known as the constant Yang-Baxter equation
\be
\lb{YB}
R_1 R_2 R_1 = R_2 R_1 R_2.
\ee
Here, we have introduced the now-standard matrix notations $R_1=R\otimes {\rm Id}_V$ and $R_2={\rm Id}_{ V}\otimes R$, where ${\rm Id}_V$ is the identity operator on the space $V$.  More generally, with each $X\in {\rm End}(V^{\otimes 2})$ one associates a sequence of operators $X_i$ 
$$
X_i:= {\rm Id}_V^{\otimes(i-1)}\otimes X\otimes  {\rm Id}_V^{\otimes(n-i-1)
}\in {\rm End}(V^{\otimes n}), \quad 1\leq i\leq n-1,
$$
acting on the spaces $V^{\otimes n}$, $n\geq 2$.

For generic  $q$, the R-matrix (\ref{R-O3}) has 3 distinct eigenvalues: $q$, $-q^{-1}$ and $q^{-2}$ and consequently, its minimal polynomial is cubic. Moreover, this R-matrix gives rise to representations of the Birman-Murakami-Wenzl algebras \cite{BW,M}. We will need several formulas following from this fact.  Details of their derivation, as well as a general description of Birman-Murakami-Wenzl algebras and of their R-matrix representations  can be found in \cite{OP-BMW,OP-reciprocal} and the references cited therein. 

Before proceeding, we make a remark regarding  notation: throughout the text, the subscript $\scriptstyle "\!R"$ in the notation $X_R$ is used to emphasize that the quantity (matrix or operation)  $X$ is uniquely constructed from a given R-matrix $R$.
A (partial) list of such quantities associated with orthogonal R-matrices is given below.

\smallskip

$\bullet$ First of all, the unnormalized projector    
$$
K_R:={ \textstyle {q^2\over \lambda\,  }}\, (q\,{\rm Id}_{V^{\otimes 2}}-R)(q^{-1}\, {\rm Id}_{V^{\otimes 2}} +R),\qquad (K_R)^2 = \eta\, K_R,\qquad \eta:=q+1+q^{-1},
$$
onto the subspace of eigenvectors of $R$ with eigenvalue $q^{-2}$ of dimension 1 (for explanation of this fact, see \cite{IOP3}). 
The explicit expression for $K$ as the product of a column vector  and a row vector is
\be
\lb{K}
K_R = 
(\,
\cdot\, \cdot\, 1\, |\, \cdot\, 1\,\cdot\,|\, {q\over v^2}\,\cdot\,\cdot\,)^{\, t}\,\,
(\,\cdot\,\cdot\, q^{-1}\,|\, \cdot\, 1\,\cdot \, |\, v^2\, \cdot\,\cdot\,).
\ee

$\bullet$ 
Secondly, in Birman–Murakami–Wenzl algebras, one can introduce a remarkable sequence of idempotents known as
{\em $q$-antisymmetrizers} (for their definition and properties, see. e.g., \cite{OP-BMW}, Sections 2.3 and 2.4). Their images in R-matrix representations (see \cite{OP-BMW}, Section 3.1) have proven to be important technical tools in the study of QM-algebras.  
These matrix images we will, with an abuse of notation, also call $q$-antisymmetrizers.
We denote these matrix projectors by $A^{(i)}_R$, $i=1, 2, \dots$, where the superscript $i$ indicates the tensor power of the spaces $V$ on which the idempotent acts: $A^{(i)}(R) \in {\rm End}(V^{\otimes i})$.

For $R$-matrices of the orthogonal type $O(n)$, the $n$-th $q$-antisymmetrizer has rank 1, while all subsequent $q$-antisymmetrizers vanish identically (see \cite{OP-reciprocal}, Sections 3.1, 3.2). Below we present formulas for the nonzero $q$-antisymmetrizers corresponding to the $O(3)$ type R-matrix (\ref{R-O3})
\ba
\lb{A2}
&&
A^{(1)}_R={\rm Id}_V, \quad A^{(2)}_R = { \textstyle {q^3\over (1+q)(1+q^2)  }}\, (q\,{\rm Id}_{V^{\otimes 2}}-R)(q^{-2}\, {\rm Id}_{V^{\otimes 2}} -R),
\\[5pt]
\lb{A3}
&&
\left(A^{(3)}_R\right)_{i_1 i_2 i_3}^{j_1 j_2 j_3}= {|\epsilon_{\scriptscriptstyle R}\rangle^{j_1 j_2 j_3}\otimes  \langle\epsilon_{\scriptscriptstyle R}|_{i_1 i_2 i_3} \over{\scriptstyle (1+q^2)(1+q+q^2)}}, 
\; \mbox{\small where}\;\;
\begin{array}{c|ccccccc|}
i\, j\, k & 123&132&213&222&231&312&321
\\
\hline	
|\epsilon_{\scriptscriptstyle R}\rangle^{i j k} & 1&-qv &\! -qv&q(q-1)v\!&qv^3&q v^3& \! -q^2 v^{4^{\phantom I}}\!\!\!
\\
\hline
\langle \epsilon_{\scriptscriptstyle R}|_{ijk} & 1&-{q\over v}&-{q\over v}&{1-q\over v}&{q\over v^3}&{q\over v^3}& \! -{q^{2^{\phantom I}}\over v^4}\!\!\!
\end{array}\; .
\ea
Here, we write the rank-1 matrix  $A^{(3)}_R$ as a tensor product of  a column 3-vector and a row 3-vector, and  list the non-zero components of these vectors 
\footnote{Here, the summation is understood to run from the lower index to the upper index.  For instance, 
	if $A$ and $B$ are two endomorphisms of a vector space $V$, then
	$(AB)^i_j:=\sum_l A^i_l B^l_j$.}.

$\bullet$ Furthermore, there exist two invertible matrices $C_R,\;D_R\in {\rm Aut}(V)$, uniquely characterized by their properties
$$
{\rm Tr}_{(1)} (C_R)_1 R_1 = ({\rm Id}_V)_2, \qquad {\rm Tr}_{(2)} (D_R)_2 R_1 = ({\rm Id}_V)_1.
$$
Here the symbol ${\rm Tr}_{(i)}$ stands for the trace over an $i$-th component space in the tensor power of the spaces $V$. 
Explicit expressions for the matrices $C_R$ and $D_R$ corresponding to the R-matrix (\ref{R-O3}) are given by:
\be
\lb{Dr}
D_R={\rm diag}\left(q^{-3},\,q^{-2},\,q^{-1}\right) ,\qquad 
C_R= {\rm diag}\left(q^{-1},\, q^{-2},\,q^{-3}\right) .
\ee
These matrices allow for the introduction of the concepts of left and right R-traces, see, e.g., \cite{Og}.
In the following, we will use the right version of the R-trace:
\be
\lb{RTr}
{\rm Tr}_R X := {\rm Tr} (D_R X), X\in {\rm End}(V)\otimes W,
\ee
where $W$ is the linear space to which the components of the matrix $X$ belong. 
\smallskip

$\bullet$ Finally, there are two more invertible matrices $O_R,\, G_R \in {\rm Aut}(V)$ that are uniquely determined by the orthogonal R-matrix. These matrices are responsible for the behavior of two remarkable group-like elements of the 
RTT-algebra associated with the given R-matrix: the  determinant and the contraction of the quantum matrix (see the next Section). Their general formulas  and explicit expressions in case of the R-matrix (\ref{R-O3}) are given below. 
\ba
\lb{G}
G_R &:=& \tr_{(2)} P_{1} (K_R)_{1} = {\rm diag}\left(v^{-2},\, 1,\,v^{2}\right) ,
\\
\lb{O} 
O_R &:=&{ q\,(q^2+1+q^{-2})\over (q^2-q+1)}\,\tr_{(2,3)}\left(P_{1}P_{2}\,A^{(3)}_R\right) = {\rm diag}\left(v^{-3},\, 1,\,v^{3}\right).
\ea
Here $P\in {\rm Aut}(V^{\otimes 2})$ is the permutation operator: $P (x\otimes y)=y\otimes x\;\;\forall\, x,y\in V$.
Formulas (\ref{G}) and (\ref{O}) are variants of the general definitions of the matrices $G_R$ and $O_R$ given in Propositions 4.6 and 4.12 of \cite{OP-reciprocal}, adapted to our specific case of the RTT-algebra and calculated for the R-matrix (\ref{R-O3}).
Note that  $G_R$ and $O_R$ satisfy the relation $(G_R)^3=(O_R)^2$, which reflects their general property proven in Proposition 4.21 \cite{OP-reciprocal}. 

\section{General results on the structure of orthogonal RTT-algebras}
\lb{RTT-alg}

In this Section, we adapt general structural results concerning orthogonal type $O(n)$ QM-algebras obtained in \cite{OP-reciprocal, OP-ortho}  to the case of the standard RTT-algebra of type $O(3)$.
\smallskip

Recall that the RTT-algebra ${\cal F}(R)$ associated with a given R-matrix $R\in {\rm Aut}(V^{\otimes 2})$\footnote{Strictly speaking, in the definition of a QM-algebra, an additional condition of strict skew-invertibility must be imposed on the R-matrix $R$. For details, see \cite{OP-BMW}, Section 4.} is the quotient  of the free associative unital 
algebra $W={\Bbb C}\langle T^a_b\rangle_{a,b=1}^{\dim V}$ 
by the two-sided ideal generated by 
the elements of  matrix relation
\be 
R\, T_{ 1}T_{ 2} = T_{ 1}T_{ 2}R\, . 
\label{RTT}
\ee
Here $T = \|T^a_b\|_{a,b=1}^{\dim V}$ is the matrix of 
generators; $T_1=T\otimes {\rm Id}_V$ and $T_2={\rm Id}_V\otimes T$ are two copies of $M$.
\smallskip

Each QM-algebra contains a distinguished abelian subalgebra—the so-called {\em characteristic subalgebra}. In the case of the RTT-algebra ${\cal F}(R)$ associated with an orthogonal $O(n)$ type R-matrix, the characteristic subalgebra is generated by a set of elements\footnote{These elements are well-defined under certain restrictions on the parameter $q$; see \cite{OP-ortho}, the preamble to Section 4.}
\be
\lb{elementary symm-f}
\goth{e}_0 :=1, \quad
\goth{e}_i := 
\Tr{1,2,\dots ,i}(T_{1}T_2\dots T_{i}\,
A_R^{(i)}), \; \ i= 1,2,\dots ,n,
\ee
called {\em the elementary sums}, and  one additional element
\be\lb{tau}
g := 
 {\lambda\, q^n\over (q^n-1)(q^{n-1}+q)}\,\Tr{1,2} \Bigl( T_{1}T_{2}\, K_R\Bigr) ,
\ee
called {\em the contraction}. 
The last non-zero elementary sum ${\goth e}_n$ is proportional to a remarkable group-like element of the RTT algebra --- the {\em $q$-determinant} of the quantum matrix $T$ (see \cite{OP-reciprocal}, Section 4.4):
\be
\lb{q-det}
{\det}_q T= q^{n(n-1)}{\goth e}_n.
\ee

The generators of the characteristic subalgebra are not algebraically independent. They satisfy the following system of {\em reciprocal relations} (see \cite{OP-reciprocal}, Section 4.3)
\be
\lb{reciprocal} 
g^{i}\, \goth{e}_{n-i}\, =\, \goth{e}_n\, \goth{e}_{i}\, ,\quad i=1,2,\dots ,n. 
\ee

Two of these elements -- $g$ and ${\goth e}_n$ -- constructed using the rank-1 matrices $K_R$ and $A_R^{(n)}$, satisfy particularly simple commutation relations with the generators of the algebra (see \cite{OP-reciprocal}, Propositions 4.6 and 4.12):
\be
\lb{T-g-e3}
T\, g\, =\, g\, ( G_R^{-1}T G_R)\, , \qquad T\, \goth{e}_n\, =\, \goth{e}_n\, ( O_R^{-1} T O_R).
\ee

Elements of the characteristic subalgebra appear as coefficients of the powers of the quantum matrix $T$ in the characteristic identity for the algebra ${\cal F}(R)$. 
To formulate this identity, it is necessary to properly define the powers of the quantum matrix $T$. We present the definition below\footnote{For a detailed discussion of the concept of the matrix product in QM-algebras, see \cite{OP-BMW}, Section 4.}.

Consider a map $\phi_R$ which combines linearly components of any matrix 
$X = \|X^a_b\|_{a,b=1}^{\dim V}$
\be
\lb{phi} 
\phi_R(X)_1 := \Tr{2} \left( X_{ 2} R\right),
\ee
An analogue of the quantum matrix {\em $k$-th power} $T^{\overline k}$ \,
is determined by induction on $k$
\be
\lb{Tk}
T^{\overline{0}} :={\rm Id}_V\, , \quad T^{\overline{k}}\, :=\, T\cdot
\phi_R(T^{\overline{k-1}}).
\ee
Here $\cdot$ denotes the usual matrix multiplication.

Now we are ready to formulate the analogue of the {\em Cayley-Hamilton theorem}
for the orthogonal $O(n)$-type RTT-algebra ${\cal F}(R)$. Here we present this theorem for algebras of type $O(2\ell - 1)$, $\ell = 2, 3, \dots$. For algebras of type $O(2\ell)$, the characteristic identity takes a different form (see, \cite{OP-ortho}, Section 4).

\begin{theor}
The quantum matrix $T$ of 
generators of the RTT-algebra ${\cal F}(R)$
of the odd orthogonal type $O(n)$,  
$n=2\ell -1$, satisfies the matrix identity
\be
\lb{orthoCH-odd}
\sum_{j=0}^{\ell -1}(-q)^{j}\sum_{r=0}^{\ell-j-1} \left\{ T^{\overline{n-j-2r}} (g^{\ell-1}{\goth e}_j-{\goth e}_n {\goth e}_{j-1}) \, -\,T^{\overline{n-j-2r-1}} ({\goth e}_n {\goth e}_{j}-g^{\ell}{\goth e}_{j-1})\right\}(q^2 g)^{2r}
\, =\, 0\, ,
\ee
where we set ${\goth e}_{-1}=0$. 
\end{theor}
We present the characteristic identity in unfactorized form: 
formula (\ref{orthoCH-odd}) is an expansion of equality (4.61) from \cite{OP-ortho}. Factorizing this equality requires the explicit use of the element $g^{-1}$, 
which appears unavoidable when defining the notion of eigenvalues of a quantum matrix (see \cite{OP-ortho}, Theorem 4.5 and Section 5). 
On the other hand, the unfactorized version (\ref{orthoCH-odd}) better serves the purpose of this work, that is,  verifying the characteristic identity using the specific example of the RTT algebra ${\cal F}(R)$.

\section{Standard $O(3)$ type RTT-algebra}

\subsection{Ordering of the generators and linear basis }
\lb{4.1}
\renewcommand{\theequation}{\thesubsection.\arabic{equation}}

In this Section we shall write in components the set of quadratic relations in the RTT-algebra ${\cal F}(R)$ with the R-matrix $R$ given by eqs. (\ref{R-O3}), (\ref{fix-u}). In the following, we will explicitly indicate the parameters $q$ and $v$ of the standard $R$-matrix in the notation for the algebra: ${\cal F}(R)_{q,v}$.
We use following notation for the components of the quantum  matrix $T$:
\[ 
T=\left(\begin{array}{ccc}x_1&a_+&c_+\\a_-&x_2&b_+\\c_-&b_-&x_3\end{array}\right)
\]

Analyzing quadratic relations (\ref{RTT})
we choose suitable ordering for the generators: 
\be
\lb{order}
(a_+,a_-,x_1)<(c_+,c_-,x_2)<(x_3,b_+,b_-) .
\ee
The relations inside the elements in each brackets are very easy, so, the orderings inside each bracket is of no
importance; we assume them anti-alphabetic inside each bracket, for example.

This ordering allows us to express unordered quadratic monomials of  generators as linear combinations of  ordered ones with  coefficients that are polynomials (not rational functions) in variables $q^{\pm 1}$, $v^{\pm 1}$.
The ordering relations within the triples in brackets are simple: 
\ba
\lb{trip-ord-1}
&&x_1a_+=\frac{q}{v}\,a_+x_1\, ,\hspace{15mm} x_2c_+=\frac{1}{v^2}\,c_+x_2\ ,
\hspace{15mm} b_+x_3 = qv\, x_3b_+\ ,
\\[2pt]
\lb{trip-ord-2}
&&x_1a_-=qv\, a_-x_1\, ,\hspace{13.5mm} x_2c_-=v^2\, c_-x_2\, ,
\hspace{16.5mm} b_- x_3 =\frac{q}{v}\, x_3b_-\, ,
\\[2pt]
\lb{trip-ord-3}
&&a_-a_+=\frac{1}{v^2}\, a_+a_-\, ,\hspace{12mm} c_-c_+=\frac{1}{v^4}\, c_+c_-\, ,
\hspace{15mm} b_-b_+=\frac{1}{v^2}\,b_+b_-\, .
\ea
We don't specify the order within the triplets. We usually use them in the order they are given in (\ref{order}), but sometimes we rearrange them to apply other  (orthogonality or diamond) relations.
\smallskip

Ordering relations between $c_\pm$ and all other matrix elements are also simple
\ba
\lb{ord-1}
c_+x_1=\frac{v^2}{q^2}\, x_1c_+\, ,& &x_3c_+=\frac{1}{q^2v^2}\,c_+x_3\,,
\\
\lb{ord-2}
c_-x_1=\frac{1}{q^2v^2}\,x_1c_-\,,&&x_3c_-=\frac{v^2}{q^2}\,c_-x_3\,,
\\
\lb{ord-3}
c_+a_+=\frac{v}{q}\,a_+c_+\,,&&b_+c_+=\frac{1}{qv}\,c_+b_+\,,
\\
\lb{ord-4}
c_-a_+=\frac{1}{qv^3}\,a_+c_-\,,&&b_+c_-=\frac{v^3}{q}\,c_-b_+\,,
\\
\lb{ord-5}
c_+a_-=\frac{v^3}{q}\,a_-c_+\,,&&b_-c_+=\frac{1}{qv^3}\,c_+b_-\,,
\\
\lb{ord-6}
c_-a_-=\frac{1}{qv}\,a_-c_-\,,&&b_-c_-=\frac{v}{q}\,c_-b_-\, .
\ea

Remaining ordering relations are
\ba
\lb{rem-ord-1}
x_2a_+=\frac{1-\lambda}{v}\,a_+x_2\,-\,\frac{\lambda}{v}\,x_1b_+\ ,&&
x_2a_-=v(1-\lambda)\,a_-x_2\,-\,\frac{\lambda v}{q}\,x_1b_-\ ,\\
\lb{rem-ord-2}
b_+x_2=v(1-\lambda)\,x_2b_+ \,-\,\frac{\lambda v}{q}\,a_+x_3\ ,&&
b_-x_2=\frac{1-\lambda}{v}\,x_2b_-\, -\,\frac{\lambda}{v}\,a_-x_3\ ,\\
\lb{rem-ord-3}
b_+x_1=\frac{v}{q}(1+\lambda)\,x_1b_+ \,+\,\frac{\lambda v}{q}\,a_+x_2\ ,&&
b_-x_1=\frac{1+\lambda}{qv}\,x_1b_- \,+\,\frac{\lambda }{v}\,a_-x_2\ ,\\
\lb{rem-ord-4}
x_3a_+=\frac{1+\lambda}{qv}\,a_+x_3 \,+\,\frac{\lambda }{v}\,x_2b_+\ ,&&
x_3a_-=\frac{v}{q}(1+\lambda)\,a_-x_3 \,+\,\frac{\lambda v}{q}\,x_2b_-\ ,
\ea
\vspace{-7mm}
\ba
\lb{rem-ord-5}
&&x_2x_1=x_1x_2\,-\,\frac{\lambda}{v}\,a_+a_-\ ,\hspace{15mm}
x_3x_2=x_2x_3\,-\,\frac{\lambda}{v}\,b_+b_-\ ,\\[3pt]
\lb{rem-ord-6}
&&x_3x_1=x_1x_3\,+\,\lambda(q-1)(x_2^2-x_1x_3)\,-\,\lambda(q-3+q^{-1})\,a_+b_-\ ,\\[3pt]
\lb{rem-ord-7}
&&b_+a_+=a_+b_+ \,-\,\frac{\lambda}{v}\,c_+x_2\ ,\hspace{15mm}
b_-a_-=a_-b_- \,-\,\lambda v\,c_-x_2\ ,\\[3pt]
\lb{rem-ord-8}
&&b_+a_-=v^2(q+\lambda(1-q))a_-b_+\, +\,\lambda v^2(x_2^2-x_1x_3)\ ,\\[3pt]
\lb{rem-ord-9}
&&b_-a_+=\frac{1}{v^2}(q+\lambda(1-q))a_+b_- +\frac{\lambda q}{v^2}(x_2^2-x_1x_3)\ .
\ea

\noindent
In total, we have $36=9\cdot 8/2=(5+1)\cdot 3+3\cdot (5+1)$ 
ordering relations. Indeed, 
the R-matrix (\ref{R-O3}) constitutes a decomposition of the space $V\otimes V$ into direct sum of subspaces of dimensions 1, 3, and 5, corresponding to the eigenvalues of $R$, respectively, $q^{-2}$, $-q^{-1}$ and $q$.
A pair of subspaces, of dimensions 5 and 1, in combination gives a symmetric subspace in $V^{\otimes 2}$, while the subspace of dimension 3 represents an antisymmetric component of $V^{\otimes 2}$.\vspace{3pt}

In addition, we have $10=5\cdot 1+1\cdot 5$ {\em orthogonality relations}. These relations are contained within invariant subspaces of ${\rm End}(V\otimes V)$ that arise from the interaction between the rank-1 and rank-5 components of the symmetrizer:
\ba
\lb{ort-1}
a_+^2=-\displaystyle{\frac{1+q}{q}\, x_1c_+\ ,}&&
a_-^2=-\displaystyle{\frac{1+q}{q^2}\,x_1c_-\ ,}\\
\lb{ort-2}
b_+^2=-\displaystyle{\frac{1+q}{q^2}\,c_+x_3\ ,}&&
b_-^2=-\displaystyle{\frac{1+q}{q}\,c_-x_3\ ,}\\
\lb{ort-3}
a_-c_+=-\frac{1}{v}\,x_1b_+-\frac{1}{v} a_+x_2\ ,&&
a_+c_-=-v\,x_1b_- \,-\,qv\,a_-x_2\ ,\\
\lb{ort-4}
c_-b_+=-\frac{1}{v}\,x_2b_- \,-\,\frac{1}{v}\, a_-x_3\ ,&&
c_+b_-=-qv\,x_2b_+ \,-\,v\,a_+x_3\ ,\\
\lb{ort-5}
a_-b_+=\displaystyle{\frac{1}{q}\,a_+b_-\ ,}&&
c_+c_-=(2-q)v^2\,a_+b_-\,+\,qv^2(x_2^2-x_1x_3)\ .
\ea
\begin{rem}{\rm
Note that for the algebra of coordinate functions on the classical orthogonal group $O(3)$, there exist 6 orthogonality relations, which follow from the condition $T^t \Omega T = \Omega$ for a symmetric $3 \times 3$ matrix. Here, $T^t$ is the transpose of the matrix $T$, and $\Omega$ is the matrix of a symmetric quadratic form.  We recall that for the quantization we use the symmetric form 
$\footnotesize{\Omega=\left(\begin{array}{ccc}0&0&1\\0&1&0\\1&0&0\end{array}\right)}$. 
Equivalently, the condition $T^t \Omega T = \Omega$ can be written as $T \Omega T^t = \Omega$. This equivalence stems from the invertibility of the orthogonal matrix $T$ and the property $\Omega$ squares to the identity matrix.
		
In our setting, the orthogonality relations in the classical case (i.e., when $R$ is a permutation matrix) take the form $T^t \Omega T = g \Omega$ and $T \Omega T^t = g \Omega$, where $g$ is the contraction defined in (\ref{tau}).  These two matrix relations are independent, provided that $g$ (and, consequently, $T$) is not invertible. Each of these matrix relations yields 5 separate equations (one fewer than 6, since  $g$ is not fixed), and together they form a total of 10 orthogonality relations.

The same phenomenon also for symplectic QM-algebras as well; see \cite{OP-ortho}, the discussion following equation (4.37).
\phantom{a}\hfill$\blacktriangleright$
}
\end{rem}
\begin{rem}{\rm It is worth noting a remarkable consistency between the ordering relations and the orthogonality relations, in the following sense. Let us apply the diamond lemma \cite{Bo,Be} to the ordering relations. Unlike in the $GL$ case (with the standard $R$-matrix $R$), nontrivial cubic relations arise here; that is, a basis of the algebra defined by the ordering relations (\ref{trip-ord-1})--(\ref{rem-ord-9}) is formed not simply by all ordered monomials, but some ordered monomials must be excluded. Remarkably, however, all the cubic relations resulting from the ordering relations turn out to be trivialized (that is, reduced to the identities $0=0$) by the orthogonality relations; moreover, all of the orthogonality relations are involved!
\phantom{a}\hfill$\blacktriangleright$
}
\end{rem}

We notice that the ordering and the orthogonality relations are invariant under an involutive ${\cal F}_{q,v}$ algebra anti-automorphism
\be
\lb{antihom}
a_\pm \leftrightarrow b_\mp, \quad c_+ \leftrightarrow c_-, \quad x_1 \leftrightarrow x_3, \quad x_2 \leftrightarrow x_2.  
\ee
It corresponds to the reflection of the quantum matrix $T$ relative to its central component $x_2$.

Moreover, there exists an involutive homomorphism on the family of algebras ${\cal F}(R)_{q,v}\leftrightarrow {\cal F}(R)_{q,v^{-1}}$, defined on the generators as
\be
\lb{hom}
a_+ \leftrightarrow a_-, \quad b_+ \leftrightarrow q^{-1} b_-, \quad c_+ \leftrightarrow q^{-1} c_-, \quad x_i \leftrightarrow x_i,\,\, i=1,2,3.  
\ee
This homomorphism corresponds to the transposition of the quantum matrix $T$.

Finally, there is a two-parameter family of scaling automorphisms of the algebra ${\cal F}(R)_{q,v}$:
\be
\lb{scale2}
a_\pm \mapsto t_1^{\pm 1} t_2\,a_\pm, \quad 
b_\pm \mapsto t_1^{\pm 1} t_2^{-1} b_\pm, \quad c_\pm \mapsto t_1^{\pm 2} c_\pm, \quad x_1 \mapsto t_2^2 x_1, 
\quad x_2 \mapsto x_2,  \quad x_3 \mapsto t_2^{-2}x_3 .  
\ee
The scaling parameters $t_1$ and $t_2$ measure the oriented distance of the generator from the main and secondary diagonals, respectively, in the quantum matrix $T$. These homomorphisms define an additive  ${\Bbb Z}^2$-grading on the algebra  ${\cal F}(R)_{q,v}$.
\ba
\lb{deg2}
{\rm grad\,}a_\pm = (\pm 1 ,1), & \quad
{\rm grad\,}b_\pm =(\pm 1 ,-1), \quad {\rm grad\,}c_\pm  = (\pm 2 ,0),\\
{\rm grad\,}x_1 =(0 ,2),& \,\,\,{\rm grad\,}x_2  =(0 ,0),  \qquad\,\, {\rm grad\,}x_3 =(0 ,-2) .  
\ea
This bi-grading is also explained by the commutation relations of $\rho:=c_+c_-$
with the generators:
$$\rho a_+=v^{-2}q^{-2}a_+\rho\ ,\ \rho a_-=v^2q^{-2}\rho\ ,\ \rho x_1=q^{-4}x_1\rho\ ,\ 
\rho c_+=v^{-4}c_+\rho\ ,\ \rho c_-=v^4 c_-\rho\ ,\ \rho x_2=x_2\rho\ ,$$
$$\rho x_3=q^4 x_3\rho\ ,\ \rho b_+=v^{-2}q^2 b_+\rho\ ,\ \rho b_-=v^2q^2b_-\rho \ .$$

Let us consider the set of cubic polynomials in generators of ${\cal F}(R)_{q,v}$. Using the ordering and the orthogonality relations one can derive following relations among seemingly irreducible (= ordered and not subject to orthogonality reduction) cubic monomials:
\ba
\lb{diam-0} 
a_- x_i b_+ &=& \frac{1}{q v^2}\,a_+ x_i b_-\ ,\qquad i=1,2,3, \\
\lb{diam-1}
a_+ a_- x_2 &=& (1-\lambda)\, a_+ x_1 b_-\, +\, \frac{v(1+q)}{q}\, x_1(x_2^2-x_1x_3)\ , \\
\lb{diam-2}
a_+ a_- b_- &=& \frac{v(1+q)}{q }\ (x_1 x_2 b_-\, +\, q v\, a_- x_1 x_3)\ , \\
\lb{diam-3}
a_+ b_+ b_- &=& \frac{v(1+q)}{q }\ (a_+ x_2 x_3\, +\,q v\, x_1 x_3 b_+)\ , \\
\lb{diam-4}
x_2 b_+ b_- &=& (1-\lambda)\, a_+ x_3 b_-\, +\, \frac{v(1+q)}{q }\ (x_2^2-x_1x_3) x_3\ ,\\[.8em]
\lb{diam-5}
x_1 b_+ b_- &=& a_+ a_- x_3\ .
\ea
We call these cubic relations the {\em diamond relations}. Let us comment on their derivation. 
\smallskip

Consider three reducible
cubic monomials $\underline{x_1 a_-} b_+$, $a_- \underline{b_+ x_3}$, $\underline{x_2 a_-} b_+$. For each one of them there are two different ways to bring it to an irreducible form. One way starts from the ordering of the underlined pairs in these monomials. In another way one uses the orthogonality relations for the pair $a_- b_+$ first. Comparing the results of the reduction one obtains equalities (\ref{diam-0}).

Relations (\ref{diam-1})--(\ref{diam-5}) are obtained in the same way, that is, by comparing results of the two different reductions of the monomials
$b_- x_1 a_+$, $a_- a_- b_+$, $a_- b_+ b_+$, $x_3 b_+ a_-$ and $x_3 x_2 x_1$, respectively.
\medskip

From now on we call {\em irreducible} the monomial that can not be transformed by an application of the ordering relations (\ref{ord-1})--(\ref{rem-ord-9}), the orthogonality relations (\ref{ort-1})--(\ref{ort-5}), the diamond relations (\ref{diam-0})--(\ref{diam-5}) and their consequent equalities (\ref{diam-0gen}) even after reshuffling its terms in the triples $(a_+, a_-, x_1)$, $(c_+, c_-, x_2)$, $(x_3,b_+,b_-)$ with the use of relations (\ref{trip-ord-1})--(\ref{trip-ord-3}). We shall arrange factors in these monomials in the order (\ref{order}). 

The complete set of the irreducible monomials in ${\cal F}(R)_{q,v}$ is given below.
\ba
\lb{list1}
1) && x_1^a\, c_\pm^b\, x_2^c\, x_3^d\,, 
\\[2pt]
\lb{list2}
2)  &&  a_\pm\,x_1^a\, c_\pm^b\, x_2^c\, x_3^d\,, 
\\[2pt]
\lb{list3}
3)  &&  x_1^a\, c_\pm^b\, x_2^c\, x_3^d\,  b_\pm\,, 
\\[2pt]
\lb{list4}
4)  &&  a_\pm\, x_1^a\, c_\pm^b\, x_2^c\, x_3^d\,  b_\pm\,,\hspace{34mm} \forall\,a,b,c,d\geq 0\, , 
\\[2pt]
\lb{list5}
5)  &&  a_+\, x_1^a\, x_2^b\, x_3^c\,  b_-\, \sim\, a_-\, x_1^a\, x_2^b x_3^c\,  b_+\,, \hspace{10mm} \forall\,a,b,c\geq 0\, , 
\\[2pt]
\lb{list6a}
6a)  &&  a_+\, a_-\, x_1^a\, x_3^{b+1}\, \sim \, x_1^{a+1}\, x_3^{b}\,  b_+\, b_-\,, \hspace{11mm} \forall\,a,b\geq 0\, ,
\\[2pt]
\lb{list6b}
6b)  &&  a_+\, a_-\, x_1^a ,\,\,\,    x_3^a\, b_+\, b_-\,,  \hspace{29mm} \forall\,a\geq 0\, .\hspace{10mm}
\ea
Here the symbol $\sim$ means equality up  to numeric factor. The equality in (\ref{list6a}) 
follows by application of the diamond relation (\ref{diam-5}) after transposition of $a_+ a_-$ with $x_1^a$. The equality in  (\ref{list5}) follows from 
\begin{lem}
\lb{lem1}
Let $Z^{(n)}$ denote arbitrary $n$-th order monomial in generators $c_\pm$ and $x_i$, $i=1,2,3$ 
The following relations hold true
	\be
	\lb{diam-0gen}
	a_- Z^{(n)} b_+ = {1\over q v^{2n}}\, a_+ Z^{(n)} b_-, \quad\,\, a_- Z^{(n)} a_+ =  {1\over v^{2(n+1)}}\, a_+ Z^{(n)} a_-.
	\ee
\end{lem}
\noindent {\bf Proof.}~
The proof is by induction on $n$.
The first relation in (\ref{diam-0gen}) generalizes eqs.(\ref{diam-0}). 
The second relation in (\ref{diam-0gen}) is required in the proof of the first one in case  if $Z^{(n)}$ contains generator $x_2$.
\phantom{a}\hfill$\blacktriangleright$\medskip

We note that the set of monomials (\ref{list1})--(\ref{list6b}) is $\sim$--invariant under the (anti-)homomorphisms (\ref{antihom}), (\ref{hom}) up to obvious transpositions (\ref{trip-ord-1})--(\ref{trip-ord-3})  
within the triples of generators.

\begin{theor} 
\lb{theor1}
Using relations (\ref{trip-ord-1})--(\ref{ort-5}) and (\ref{diam-0})--(\ref{diam-5}) a finite number of times, every monomial in generators $a_\pm$, $b_\pm$, $c_\pm$, $x_i$, $i=1,2,3$, of the algebra ${\cal F}(R)_{q,v}$ can be expressed as a linear combination of the irreducible ones (\ref{list1})--(\ref{list6b}).
\end{theor}

Before we prove the theorem,
let's derive ordering relations between certain combinations of generators of
${\cal F}(R)_{q,v}$ with powers of $x_i$ (see two lemmas and two corollaries below).
\begin{lem} Permutation relations for $x_1^n$,  $n>0$, have the form:
\lb{x1^n}
\ba
\lb{x2x1}
x_2\, x_1^n &=&  x_1^n\, x_2 \, +\,  {1-q^{2n}\over v\, q}\, a_+\, a_-\,x_1^{n-1}\, , 
\\[6pt]
\lb{bpx1}
b_+\, x_1^n &=&  {v^n(q^{2n}+q^{2n-1}-1)\over q^{3n-1}}\, x_1^n\, b_+\, +\,  {v\,(q^{2n}-1)\over q^{2n}}\, a_+\, x_1^{n-1} x_2\, ,
\\[6pt]
\nn
 x_3\, x_1^n &=& x_1^n\, x_3\,+\,{(1-q^{2n})(1-q^{2n-1})\over q^{4n-3}}\, x_1^{n-1}\,(x_2^2-x_1 x_3)
\\[1pt]
\lb{x3x1}
&&\hspace{8mm} +\, {(1-q^{2n})\left[(2-q)(1-q^{2n-1}) -q^{2n-2}\right]\over v^{n-1} q^{3n-1} }\, a_+\, 
x_1^{n-1}\, b_-\, ,
\\[6pt]
\nn
(x_2^2-x_1 x_3)\, x_1^n &=& {1+q(1+q)(2-q)(1-q^{2n})\over q^{4n}}\, x_1^n\,(x_2^2-x_1 x_3)
\\[1pt]
\lb{x123x1}
&&\hspace{8mm} +\, {\left[1+q(1+q)(2-q)\right](2-q)(1-q^{2n})\over v^n q^{3n+1} }\, a_+\, x_1^n\, b_-\, .
\ea
\end{lem}

\noindent {\bf Proof.}~
Relations (\ref{x2x1})-(\ref{x123x1}) are proved by induction on $n$.

The proof of relation (\ref{x2x1}) is a direct check, the induction base is eq.(\ref{rem-ord-5}). 

Let us rewrite relation (\ref{bpx1}) in a concise form
\be
\lb{bpx1-2}
b_+ x_1^n = v^n X_n\, x_1^n b_+ \,+\, v Y_n\, a_+ x_1^{n-1} x_2,
\ee
where $X_n$ and $Y_n$ are  polynomials in $q^{\pm 1}$. Obviously, $X_0=1$, $Y_0=0$. Relation (\ref{rem-ord-3}) states that $X_1={1+\lambda\over q}$, $Y_1= {\lambda\over q}$. This is the induction base for the proof of (\ref{bpx1}),

The induction step 
$n\rightarrow n+1$  consists of calculating $b_+ x_1^{n+1}= (b_+ x_1^n) x_1$. Applying successively the induction assumption (\ref{bpx1-2}), relations (\ref{rem-ord-3}) and (\ref{rem-ord-5}), and then using relations (\ref{trip-ord-1}), (\ref{ort-1}), (\ref{ord-5}), (\ref{ort-3}) and again (\ref{trip-ord-1}), we obtain the following recurrence relations for the coefficients $X_n$, $Y_n$:
$$
\left\{\!\!
\begin{array}{rcl}
\nn
X_{n+1} &=& q^{-1}(1+\lambda)\, X_n \,-\, \lambda\, q^{-n-1}(1+q)\, Y_n, 
\\[3pt]
\nn
Y_{n+1} &=& \lambda\, q^{n-1} X_n \, +\, (q^{-2}-\lambda)\, Y_n.
\end{array}
\right.
$$
It is now easy to verify that the coefficients in (\ref{bpx1}) satisfy these relations.

Now, let us rewrite relations (\ref{x3x1}), (\ref{x123x1}) in a concise form
\ba
\lb{x3x1-2}
x_3 x_1^n &=& x_1^n x_3\,+\,  A_n\, x_1^{n-1} (x_2^2-x_1 x_3) \,+\, v^{1-n} B_n\, a_+ x_1^{n-1} b_-,
\\[2pt]
\lb{x123x1-2}
(x_2^2-x_1 x_3)\, x_1^n &=& C_n\,  x_1^n (x_2^2-x_1 x_3)\, +\, v^{-n} D_n\, a_+ x_1^n b_-.
\ea
where $A_n$, $B_n$, $C_n$ and $D_n$ are  polynomials in $q^{\pm 1}$.  Obviously, $A_0=B_0=D_0=0$, $C_0=1$. Relation (\ref{rem-ord-6}) states that 
$$
A_1=\lambda(q-1), \quad B_1= -\lambda (q-3+q^{-1}).
$$ 

Calculation of the coefficients $C_1$, $D_1$ is a bit more laborious. To bring element $(x_2^2-x_1 x_3) x_1$ to the irreducible form one applies 
relations (\ref{rem-ord-6}) and two times (\ref{rem-ord-5}) and then uses relations (\ref{trip-ord-1}), (\ref{rem-ord-1}), (\ref{rem-ord-8}), (\ref{ort-5}) and finally (\ref{diam-1}). The calculation gives following expressions for the coefficients
$$
C_1=q^{-4}[1+q(1+q)(2-q)(1-q^2)],\quad D_1=q^{-4}{[1+q(1+q)(2-q)](2-q)(1-q^2)},
$$
which agrees with the general formula (\ref{x123x1}). Thus, we have verified the induction base for proving the relations  (\ref{x3x1}), (\ref{x123x1}).

In the proof of (\ref{x3x1}) the induction step 
$n\rightarrow n+1$  consists of calculating $x_3 x_1^{n+1}= (x_3 x_1^n) x_1$. Applying successively the induction assumption (\ref{x3x1-2}), relations (\ref{rem-ord-6}), (\ref{x123x1-2}) for $n=1$ (this relation was verified above) and (\ref{rem-ord-3}), and then using relations (\ref{trip-ord-1}) several times and (\ref{diam-1}) we obtain the following recurrence relations for the coefficients $A_n$, $B_n$:
$$
\left\{\!\!
\begin{array}{rcl}
	\nn
	A_{n+1} &=& A_1\, +\, C_1 A_n \,+\, \lambda\, q^{-n}(1+q)\, B_n, 
	\\[4pt]
	\nn
	B_{n+1} &=&  q^{n} B_1\,+\, q^{n-1}D_1 A_n \, +\, (q+{\lambda\over q}-\lambda^2 )\, B_n.
\end{array}
\right.
$$
Coefficients in formula (\ref{x3x1}) verify these relations.

The induction step $n\rightarrow n+1$ in the proof of relation (\ref{x123x1}) consists of the calculation of
$\left((x_1^2-x_1 x_3)x_1^{n} \right)x_1$. Here,  by applying the induction assumption (\ref{x123x1-2}) and relation (\ref{rem-ord-3}), and then using formulas (\ref{trip-ord-1}) and (\ref{diam-1}), we finally obtain the following recurrence relations for the coefficients $C_n$ and $D_n$:
$$
\left\{\!\!
\begin{array}{rcl}
	\nn
	C_{n+1} &=&  C_1 C_n \,+\, \lambda\, q^{-n-1}(1+q)\, D_n, 
	\\[4pt]
	\nn
	D_{n+1} &=&  q^{n} D_1  C_n \, +\, (q+{\lambda\over q}-\lambda^2 )\, D_n.
\end{array}
\right.
$$
Coefficients in formula (\ref{x123x1}) verify these relations.
\phantom{a}\hfill$\blacktriangleright$\medskip

Applying (anti)homomorphic maps (\ref{antihom}), (\ref{hom}) to relations (\ref{x2x1})-(\ref{x123x1}), one can obtain permutation relations of $x_1^n$ with $b_-$ and $a_\pm$, and of $x_3^n$ with $x_{1,2}$, $a_\pm$ and $b_\pm$. Below we present those that will be used in what follows.
\begin{cor}
\lb{cor-1}
Some of the permutation relations for $x_3^n$,  $n>0$, have the form:
\lb{x3^n}
\ba
\lb{x3x2}
x_3^n\, x_2  &=&  x_2 \, x_3^n\, +\,  {1-q^{2n}\over v\, q^{2n-1}}\, b_+\, b_-\,x_3^{n-1}\, , 
\\[6pt]
\lb{x3ap}
x_3^n\, a_+ &=&  {q^{2n}+q^{2n-1}-1\over v^n\, q^{3n-1}}\, a_+\, x_3^n\, +\,  {q^{2n}-1\over v\, q^{2n-1}}\, x_2\, x_3^{n-1} b_+\, .
\ea
\end{cor}
\medskip

The permutation relations of the generators with  $x_2^n$ are more complicated. We will derive implicit expressions for the numerical coefficients in these relations, fixing their dependence  only on the powers of $v$ and $\lambda=q-q^{-1}$. 
The explicit form of the coefficients as polynomials in $q^{\pm 1}$, $v^{\pm 1}$ is not required in our further considerations.
\begin{lem} Permutation relations for $x_2^n$, $n>1$, have the form:
\lb{x2^n}
\ba
\lb{bpx2}
b_+\, x_2^n &=&  v^{ n} \sum_{i\geq 0}\lambda^i Q_{n,i}\, x_1^i\, x_2^{n-2i}\, x_3^i\, b_+\,
 +\, v \sum_{i\geq 0}  \lambda^{i+1} T_{n,i}\, a_+\, x_1^{i} x_2^{n-2i-1} x_3^{i+1}\, ,
\\[6pt]
\lb{bmx2}
b_-\, x_2^n &=&  v^{ -n} \sum_{i\geq 0}\lambda^i Q_{n,i}\, x_1^i\, x_2^{n-2i}\, x_3^i\, b_-\,
+\, v^{-1} \sum_{i\geq 0}  q\, \lambda^{i+1} T_{n,i}\, a_-\, x_1^{i} x_2^{n-2i-1} x_3^{i+1}\, ,
\\[6pt]
\lb{x2ap}
 x_2^n\, a_+ &=& v^{- n} \sum_{i\geq 0}\lambda^i Q_{n,i}\, a_+\, x_1^i\, x_2^{n-2i}\, x_3^i\,
 +\, v^{- 1}\sum_{i\geq 0}  q\, \lambda^{i+1} T_{n,i} \,  x_1^{i+1} x_2^{n-2i-1} x_3^{i}\, b_+\, ,
\\[6pt]
\lb{x2am}
x_2^n\, a_- &=& v^{ n} \sum_{i\geq 0}\lambda^i Q_{n,i}\, a_-\, x_1^i\, x_2^{n-2i}\, x_3^i\,
+\, v\,\sum_{i\geq 0}   \lambda^{i+1} T_{n,i} \,  x_1^{i+1} x_2^{n-2i-1} x_3^{i}\, b_-\, ,
\\[6pt]
\nn
x_3\, x_2^n &=& \sum_{i\geq 0} \lambda^i  S_{n,i}\, x_1^i\, x_2^{n-2i}\, x_3^{i+1}\,
+\, v^{1-n} \sum_{i\geq 0}  \lambda^{i+1} U_{n,i}\, a_+\, x_1^{i}\, x_2^{n-2i-2}\,  x_3^{i+1}\, b_-
\\
\lb{x3x2n}
&&  \hspace{8mm}+\, 
\varepsilon_{n}\, v^{-1} \,\lambda^{n+1\over 2}\, W_n\, x_1^{n-3\over 2}\,a_+\, a_-\, \,x_3^{ {n+1\over 2} }\, ,
\\[8pt]
\nn
x_2^n\, x_1 &=& \sum_{i\geq 0} \lambda^i  S_{n,i}\, x_1^{i+1}\, x_2^{n-2i}\, x_3^{i}\,
+\, v^{1-n} \sum_{i\geq 0}  \lambda^{i+1} U_{n,i}\, a_+\, x_1^{i+1}\, x_2^{n-2i-2}\,  x_3^{i}\, b_-
\\
\lb{x2nx1}
&&  \hspace{8mm}+\, 
 \varepsilon_{n}\, v^{-1}\,\lambda^{n+1\over 2}\, W_n\, x_1^{n-1\over 2}\,a_+\, a_-\, \,x_3^{ {n-1\over 2} }\, .
\ea
Here $Q_{n,i}$, $T_{n,i}$, $S_{n,i}$, $U_{n,i}$ and $W_n$ are some polynomials in $q^{\pm 1}$, $\varepsilon_n$ equals $0 / 1$ if $n$ is even$/$odd.

For the lowest coefficients in the expansions  of the right hand sides of (\ref{bpx2})--(\ref{x3x2}) in powers of $\lambda$ the explicit expressions are
\be
Q_{n,0} = (1-\lambda)^n, 
\qquad S_{n,0}=(q^{-2}-\lambda)^n\,\, \forall\, n>1. 
\ee
\end{lem}
\begin{rem} {\rm
The calculation of the numerical coefficients in formulas (\ref{bpx2})--(\ref{x3x2n}) for $n\leq 6$ shows that $Q_{n,i}$ and $T_{n,i}$ are divisible by $\lambda^i$, and $S_{n,i}$  is divisible by $(1+q)$. We do not verify these fine estimates here.
}
\phantom{a}\hfill$\blacktriangleright$

\end{rem}
 Before proceeding to the proof of the lemma \ref{x2^n} we will deduce its corollary.
\begin{cor}
\lb{cor-2}
Assuming validity of the relations  (\ref{x2ap}) for $x_2^m a_+$  and (\ref{x2nx1}) for $x_2^m x_1$   for all $ m=1,\dots, n-1$,
the reduced expression for the monomials $x_2^n b_+ b_-$, $n>0$, is
\ba
\nn
x_2^n\, b_+\, b_- &=& v\, q^{-1}(1+q)\, x_2^{n+1}\, x_3\, +\, v \sum_{i\geq 0} \lambda^i M_{n,i}\,
x_1^{i+1}\, x_2^{n-2i-1}\, x_3^{i+2}
\\[2pt]
\lb{x2bpbm}
&& +\, v^{1-n}\sum_{i\geq 0} \lambda^i J_{n,i}\, a_+\, x_1^i\, x_2^{n-2i-1}\, x_3^{i+1}\, b_-\, +
\varepsilon_{n-1}\lambda^{n\over 2} L_n\, x_1^{{n\over 2}-1} a_+\, a_-\, x_3^{{n\over 2}+1}\, .
\ea
Here $M_{n,i}$, $J_{n,i}$ and $L_n$ are some polynomials in $q^{\pm 1}$. They can be expressed in terms of
$Q_{m,i}$, $T_{m,i}$, $S_{m,i}$, $U_{m,i}$, $W_m$, $m=1,2,\dots , n-1$. Explicit expression for $J_{n,0}$ is
\be
\lb{J0}
J_{n,0} = (1-\lambda)\, Q_{n-1,0} = (1-\lambda)^n .
\ee
\end{cor}

\noindent {\bf Proof.}~ Let us calculate $x_2^n b_+ b_-$:
$$
x_2^n\, b_+\, b_-\, =\, x_2^{n-1} (x_2 b_+ b_-)\, =\, v\, {\textstyle { 1+q\over q}}\, x_2^{n+1}\, x_3\, +\,
 (1-\lambda)\, (x_2^{n-1} a_+)\, x_3 b_-\, -\, v\, {\textstyle { 1+q\over q}}\, (x_2^{n-1} x_1)\, x_3^2
$$
Here we used the diamond relation (\ref{diam-4}) at the beginning of the transformation. 
The last two terms in the right hand side of this calculation can be transformed  further with the use of
relations (\ref{x2ap}) and (\ref{x2nx1}). Using eqs.(\ref{trip-ord-1}), (\ref{trip-ord-2}), (\ref{diam-5}) and collecting similar terms
we finally obtain
\ba
\nn
x_2^n\, b_+\, b_- &=&  v\, {\textstyle { 1+q\over q}}\, x_2^{n+1}\, x_3\, +\, (1-\lambda) \sum_{i \geq 0}
\lambda^{i+1} q^{-2i} T_{n-1,i}\, x_1^{i+1}\, \underline{ x_2^{n-2(i+1)} b_+ b_-}\, x_3^{i+1}
\\[2pt]
\nn
&& -\, v\, {\textstyle { 1+q\over q}}\, \sum_{i\geq 0}\lambda^i S_{n-1,i}\, x_1^{i+1} x_2^{n-2i-1}\, x_3^{i+2}
\\[2pt]
\nn
&& 
+\, v^{1-n}\sum_{i\geq 0} \lambda^i \Bigl((1-\lambda) Q_{n-1,i}-q(1+q) U_{n-1,i-1}\Bigr)\, a_+\, x_1^i x_2^{n-2i-1}\, x_3^{i+1}\, b_-
\\[2pt]
\nn
&& 
-\, \varepsilon_{n-1}\, {\textstyle { 1+q\over q}}\,\lambda^{n\over 2}\, W_{n-1}\, x_1^{{n\over 2}}\, b_+\, b_-\, x_3^{{n\over 2}}.
\ea
In this expression the only reducible term is the rightmost one in the first line. This term contains the underlined combination $x_2^{n-2(i+1)} b_+ b_-$, which is similar to  $x_2^{n} b_+ b_-$, but has a lower exponent of the generator $x_2$: $n-2(i+1) <n$.  Now, by applying induction on $n$ we can verify the relation (\ref{x2bpbm}) and the statements about the form of the coefficients including (\ref{J0}). The base for induction is the diamond relation (\ref{diam-4}).
\phantom{a}\hfill$\blacktriangleright$\medskip

\noindent {\bf Proof of Lemma \ref{x2^n}.}~ Among six relations (\ref{bpx2})--(\ref{x2nx1}), only two, say (\ref{bpx2}) and (\ref{x3x2n}), are independent. The other four follow from this pair by applying  (anti)homomorphisms
(\ref{antihom}) and (\ref{hom}). 

\vspace{2pt}
We will start by considering the relation (\ref{bpx2}). The proof goes by induction on $n$, the base for induction is the relation (\ref{rem-ord-2}). The induction step $n\rightarrow n+1$ consists of the following calculation 
\be
\lb{step-1}
b_+\, x_2^{n+1}\, =\, (b_+\, x_2^{n}) x_2\, =\, {l.c.} \bigl\{ v^n\lambda^i\, x_1^i x_2^{n-2i} x_3^i\, b_+ x_2\, ,\,
v\, \lambda^{i+1}\, a_+ x_1^i x_2^{n-2i-1} x_3^{i+1} x_2 \bigr\}_{i\geq 0}.
\ee
Here we apply the induction assumption and use an abbreviated notation: ${l.c.}\bigl\{ \dots \bigr\}$, which means a linear combination of the terms in braces, the coefficients of which are polynomials in $q^{\pm 1}$.
\smallskip

Consider transformation of the typical monomial from the first subset in the braces in (\ref{step-1}).
Applying relations (\ref{rem-ord-2}) and then, (\ref{x3x2}) and (\ref{x3ap}), we obtain
\ba
\nn
v^n \lambda^i\, x_1^i x_2^{n-2i} x_3^i\, b_+ x_2 &=&  {l.c.} \bigl\{ v^{n+1}\lambda^i\, x_1^i x_2^{n+1-2i} x_3^i\, b_+ \, , \, 
v^n \lambda^{i+1} \, x_1^i x_2^{n-2i} \underline{x_3^{i-1} b_+ b_- b_+} \, ,\,
\\[3pt]
\lb{step-2}
&& 
\hspace{7mm}
v^{n+1-i} \lambda^{i+1} \, x_1^i x_2^{n-2i}a_+ x_3^{i+1} 
\, , \, v^n \lambda^{i+2} x_1^i x_2^{n-2i+1} x_3^{i-1} b_+ x_3 \bigr\}\, .
\ea
Here, the first term in braces is irreducible. It corresponds  to the left term in the right hand side of the formula
(\ref{bpx2}) we are testing. Upon permuting $b_+$ and $x_3$, the last term becomes identical to the first one. 

Next, we transform the third term in the braces in (\ref{step-2}).
\ba
\nn
v^{n+1-i} \lambda^{i+1} \, x_1^i (x_2^{n-2i}a_+) x_3^{i+1} \, =\,  {l.c.} \bigl\{ v^{i+1}\lambda^{i+j+1}\, x_1^i (a_+ x_1^j x_2^{n-2(i+j)} x_3^j)x_3^{i+1}\, , \hspace{30mm}
\\[2pt]
\nn
v^{n-i} \lambda^{i+j+2} \, x_1^i (x_1^{j+1} x_2^{n-2(i+j)-1} x_3^j b_+) x_3^{i+1}\bigr\}_{j\geq 0}\hspace{7mm}
\\[4pt]
\lb{step-3}
=\, {l.c.} \bigl\{ v\,\lambda^{i+j+1}\, a_+ x_1^{i+j} x_2^{n-2(i+j)} x_3^{i+j+1}\, , \,
v^{n+1} \lambda^{i+j+2} \, x_1^{i+j+1} x_2^{n+1-2(i+j+1)} x_3^{i+j+1}b_+\bigr\}_{j\geq 0}\, .\hspace{5mm}
\ea
Here, we substituted formula (\ref{x2ap}) for $x_2^{n-2i} a_+$ (thus, applying the induction assumption),
and then, used permutation relations in the pairs $a_+$, $x_1$ and $b_+$, $x_3$ to order them. We see, that the second term in braces in (\ref{step-3}) again reproduces the first term in the r.h.s. of the formula (\ref{bpx2}), while the first term in (\ref{step-3}) corresponds to the second terms in the r.h.s. of  (\ref{bpx2}).
\smallskip

Finally, we shall bring the second term in the braces in (\ref{step-2}) to the irreducible form. To this end we transform its last underlined in (\ref{step-2}) part:
\ba
\lb{step-4}
x_3^{i-1} b_+ b_- b_+ \, \sim \, v^{-i-1}\, b_+^2 b_- x_3^{i-1}\, \sim \,  v^{-i}\, c_+ b_- x_3^i\, =\,
{l.c.}\bigl\{ v\, x_2 x_3^i b_+\, ,\, v^{1-i}\, a_+ x_3^{i+1}\bigr\}\, .
\ea
Here symbol "$\sim$" means an equality up to multiplication by a polynomial in $q^{\pm 1}$ factor. In the transformation (\ref{step-4}), we used permutations in the triple $x_3$, $b_\pm$ in the first step, the orthogonality relations (\ref{ort-2}) and (\ref{ort-4}) in the second and third steps, and completed the transformation, rearranging $x_3$ and $b_+$ again. Substituting the result of the transformation (\ref{step-4}) back into the second term in (\ref{step-2}) we find that this term is a linear combination of the terms like the first and the third ones in (\ref{step-2}).

Thus, we confirmed that the first subset in braces in (\ref{step-1}) reduces to  a linear combination of the irreducible terms which are present in the formula (\ref{bpx2}).
\smallskip

It remains  to transform monomials from the second subset in the braces in (\ref{step-1}). 
We apply relation (\ref{x3x2})
\ba
\lb{step-5}
v\, \lambda^{i+1} \, a_+ x_1^i x_2^{n-2i-1} (x_3^{i+1} x_2)\, =\,  {l.c.} \bigl\{ v\,\lambda^{i+1}\, a_+ x_1^i x_2^{n-2i}x_3^{i+1}\, ,\,
\lambda^{i+2} \, a_+ x_1^i  x_2^{n-2i-1} b_+ b_- x_3^{i}\bigr\} 
\ea

The first term in the r.h.s of (\ref{step-5}) is irreducible. It has a form of the second term in (\ref{bpx2}). 

The second term in the r.h.s. of (\ref{step-5}) 
is reducible. Applying relations (\ref{x2bpbm}) we continue its transformation
\ba
\nn
\lambda^{i+2} \, a_+ x_1^i (x_2^{n-2i-1} b_+ b_-) x_3^{i} \, =\,  {l.c.} \bigl\{ v\,\lambda^{i+2}\, a_+ x_1^i x_2^{n-2i} x_3^{i+1}\, , 
v\, \lambda^{i+j+2}\,a_+ x_1^{i+j+1} x_2^{n-2(i+j+1)} x_3^{i+j+2}\, ,
\\[4pt]
\lb{step-6}
v^{2(i+1)-n}\lambda^{i+j+2}\, a_+ x_1^{i} a_+ x_1^{j} x_2^{n-2(i+j+1)} x_3^{j+1} b_- x_3^i\, , \,
\varepsilon_n\, \lambda^{n+3\over 2} \,a_+\, x_1^{n-3\over 2} a_+ a_- x_3^{n+1\over 2}\bigr\}_{j\geq 0}\, .\hspace{3mm}
\ea
The first two terms in the r.h.s. of (\ref{step-6}) are irreducible and have a form of the second term in (\ref{bpx2}).
Continuing transformation of the last term in the r.h.s. of (\ref{step-6}) we obtain
\ba
\nn
\varepsilon_n \lambda^{n+3\over 2}\, a_+ x_1^{n-3\over 2} a_+ a_- x_3^{n+1\over 2}\, =\,
\varepsilon_n \lambda^{n+3\over 2}\, v^{n-3\over 2} x_1^{n-1\over 2} c_+ a_- x_3^{n+1\over 2}\hspace{15mm}
\\[3pt]
\lb{step-7}
=\, {l.c.} \bigl\{ \varepsilon_n v^{n+1}\lambda^{n+3\over 2}\, x_1^{n+1\over 2} x_3^{n+1\over 2} b_+ \, , \, \varepsilon_n
v\,  \lambda^{n+3\over 2} \, a_+ x_1^{n-1\over 2} x_2 x_3^{n+1\over 2} \bigr\}\, .
\ea
In the first step of this transformation we permuted $a_+$ with $x_1$ and applied orthogonality relation (\ref{ort-1}).
Then we permuted $c_+$ with $a_-$, used the orthogonality relation (\ref{ort-3}) and finally, permuted $b_+$ with $x_1$, and $x_1$ with $a_+$.
As we see, all the terms in (\ref{step-7}) are irreducible and they are present in the r.h.s. of (\ref{bpx2}).

It remains to consider the third term in the r.h.s. of (\ref{step-6}). The transformation goes as follows
\ba
\nn
v^{2(i+1)-n}\lambda^{i+j+2} \, a_+ x_1^i a_+ x_1^j x_2^{n-2(i+j+1)}  x_3^{j+1} b_- x_3^i \, =\,   v^{n-k-1}\lambda^{k+2}\, x_1^{k+1} x_2^{n-2(k+1)} c_+ b_- x_3^{k+1} |_{k=i+j}\, 
\\[4pt]
\lb{step-8}
=\, {l.c.}\bigl\{
v^{n+1}\lambda^{k+2}\, x_1^{k+1} x_2^{n-2(k+1)} x_3^{k+1} b_+\, , \,
v^{n+1-(k+1)}\, \lambda^{k+2} \, x_1^{k+1} x_2^{n-2(k+1)}a_+  x_3^{k+2}\bigr\}|_{k=i+j}\, .\hspace{5mm}
\ea
Here in the first line we moved $b_-$ to the left of $x_3$ and $a_+$ to the right of $x_1$, applied orthogonality relation
(\ref{ort-1}) and moved $c_+$ to the right of $x_2$. In the second line we used orthogonality relation (\ref{ort-4}) and then, permuted $b_+$ and  $x_3$ in the first term. The first term in (\ref{step-8}) becomes irreducible. It has a form of the first term in the r.h.s. of (\ref{bpx2}). The second term we already transformed to an irreducible form in the calculation (\ref{step-3}). The induction step for the relation (\ref{bpx2}) is completed.
\smallskip

The proof of the induction step  $n\rightarrow n+1$ for the relation (\ref{x3x2n}) is similar. We will briefly comment on it.
The induction step $n\rightarrow n+1$ consists in the following calculation 
\ba
\nn
x_3\, x_2^{n+1}\, =\, (x_3\, x_2^{n}) x_2\, =\, {l.c.} \bigl\{ \lambda^i\, x_1^i x_2^{n-2i} x_3^{i+1}\, x_2\, ,\,
v^{i+2-n}\, \lambda^{i+1}\, a_+ x_1^i x_2^{n-2i-2} b_- x_3^{i+1} x_2 \, ,\,
\\[4pt]
\lb{step-21}
\varepsilon_n\, v^{-1} \lambda^{n+1\over 2}\, x_1^{n-3\over 2} a_+ a_- x_3^{n+1\over 2} x_2
\bigr\}_{i\geq 0}\, .\hspace{1mm}
\ea
Here we applied the relation (\ref{x3x2n}) (our induction assumption) and rearranged $b_-$ and $x_3$ in the second term on the right hand side.

The first term in the r.h.s. of (\ref{step-21}) can be transformed using the relations (\ref{x3x2}) and (\ref{x2bpbm}). 
Its irreducible expression contains all terms that are present in the r.h.s. of (\ref{x3x2n}). 

The last term in the r.h.s. of (\ref{step-21}) can be brought to an irreducible form with the help of orthogonality relations
(\ref{ort-1}), (\ref{ort-2}), (\ref{ort-5}) and the diamond relation (\ref{diam-1}). Its irreducible expression contains first and second terms from the r.h.s. of (\ref{x3x2n}). 

We will take a closer look at the most time-consuming calculation of the second term  in the r.h.s. of (\ref{step-21}). To transform this term we use relation (\ref{x3x2}) and then,  permute generators  $b_-$ and $x_2$ with the help of relation (\ref{rem-ord-2}). The result is
\ba
\nn
v^{i+2-n}\, \lambda^{i+1}\, a_+ x_1^i x_2^{n-2i-2} b_- x_3^{i+1} x_2 \, =
 {l.c.}\bigl\{
 v^{-n} \lambda^{i+1} a_+ x_1^i x_2^{n-2i-1} x_3^{i+1} b_-\, ,\hspace{15mm}
\\[4pt]
\lb{step-22}   
v^{i+1-n} \lambda^{i+2} a_+ x_1^i x_2^{n-2i-2} a_- x_3^{i+2}\, ,\,
v^{1-n} \lambda^{i+2} a_+ x_1^i x_2^{n-2i-2} x_3^i b_- b_+ b_-
\}\, .
\ea
The first term in the r.h.s of (\ref{step-22}) is irreducible. It has a form of the second term in (\ref{x3x2n}).

To bring the second term in the r.h.s of (\ref{step-22}) to the irreducible form we apply consecutively
the relation (\ref{x2am}) and the image of the relation (\ref{x2bpbm}) under the antihomomorphism (\ref{antihom}) which gives the reduced expressions for the monomials $a_+ a_- x_2^n$. At the final stage of the transformation we rearrange pairs of generators $x_1$, $a_+$ and $b_-$, $x_3$ and use the diamond relations (\ref{diam-5}). The resulting expression contains irreducible terms from the r.h.s. of (\ref{x3x2n}).

The last term in the r.h.s. of (\ref{step-22}) contains the monomial $x_3^i b_- b_+ b_-$. Applying the homomorphic map (\ref{hom}) to the relation (\ref{step-4}) we obtain the reduced expression for it
\ba
\lb{step-23}
x_3^{i} b_- b_+ b_- \, = \, 
{l.c.}\bigl\{ v^{-1} x_2 x_3^{i+1} b_-\, ,\, v^{i}\, a_- x_3^{i+2}\bigr\}\, .
\ea
Therefore for the last term in the r.h.s. of (\ref{step-22}) we have
\ba
\nn
v^{1-n} \lambda^{i+2} a_+ x_1^i x_2^{n-2i-2} x_3^i\, b_- b_+ b_-\, =
{l.c.}\bigl\{
v^{-n} \lambda^{i+2} a_+ x_1^i x_2^{n-2i-1} x_3^{i+1} b_-\, ,\hspace{15mm}
\\[4pt]
\lb{step-22}   
v^{i+1-n} \lambda^{i+2} a_+ x_1^i x_2^{n-2i-2} a_- x_3^{i+2}
\}\, .
\ea
Thus, we expressed the last term on the r.h.s. of the relation (\ref{step-22}) as a linear combination of the first two terms from there.  Hence, its reduced expression contains  terms from the r.h.s. of (\ref{x3x2n}) only.
The proof of the induction step for the relation (\ref{x3x2n}) is finished.
\phantom{a}\hfill$\blacktriangleright$\smallskip

\noindent {\bf Proof of Theorem \ref{theor1}.}~
Verifying the statement of the theorem is a straightforward although lengthy calculation.
It goes by induction on the degree of the monomials in the algebra generators. 
\smallskip

The base of the induction is the set of second-degree monomials. Indeed, the $9$ generators of the algebra ${\cal F}(R)_{q,v}$ form  $9^2 =81$ quadratic combinations, which are subject to $36$ ordering and $10$ orthogonality relations. 
All remaining $35 = 9^2- 36 -10$ independent quadratic monomials are present in the list (\ref{list1})--(\ref{list6b}).

The inductive step consists of checking that monomials of degree $n+1$ obtained by multiplying all irreducible monomials of degree $n$ by the algebra generators on the left
(or, equally, on the right) can be expressed as  linear combinations of the degree $n+1$ irreducible monomials from the list (\ref{list1})--(\ref{list6b}). 
This calculation goes with the use of the permutation relations for $x_i^n$, that we derived in the lemmas \ref{x1^n},
\ref{x2^n} and the corollaries \ref{cor-1}, \ref{cor-2}. We leave it to the interested reader.
\hfill$\blacksquare$\medskip

Let us place all irreducible cubic monomials on a square lattice: the coordinate of a monomial on the lattice coincides with its ${\Bbb Z}^2$-grading. We get the following figure:
$$
\begin{array}{lllllll}
\hspace{1mm}	{\color{blue}(0,6)\atop \bullet}\, \mathsmaller{x_1^3} & \hspace{-3mm}	
	{\color{blue}(1,5)\atop \bullet}\,\,\mathsmaller{ a_+ x_1^2} & 	\hspace{-4mm}
	{\color{blue}(2,4)\atop \bullet}\,\,\,\mathsmaller{x_1^2 c_+} & \hspace{-4mm}	
	{\color{blue}(3,3)\atop \bullet}\,\,\,\mathsmaller{a_+ x_1 c_+}& 	\hspace{-4mm}
	{\color{blue}(4,2)\atop \bullet}\,\,\,\mathsmaller{x_1 c_+^2} &  \hspace{-4mm}	\,
	{\color{blue}(5,1)\atop \bullet}\,\,\,\mathsmaller{a_+ c_+^2} &\hspace{-3mm}  	
	{\color{blue}(6,0)\atop \bullet}\,\,\,\mathsmaller{c_+^3}
	\\[14pt]
	{\color{blue}(-1,5)\atop \bullet}\,\mathsmaller{a_- x_1^2} &  \hspace{-3mm}
	{\color{blue}(0,4)\atop \bullet} \,\begin{array}{l}\mathsmaller{x_1^2 x_2}\\[-4pt] \mathsmaller{a_+a_-x_1} \end{array} & 	\hspace{-4mm}
	{\color{blue}(1,3)\atop \bullet} \,\begin{array}{l}\mathsmaller{a_+ x_1 x_2} \\[-4pt] \mathsmaller{x_1^2 b_+}\end{array} &  \hspace{-4mm}
	{\color{blue}(2,2)\atop \bullet} \,\,\begin{array}{l}\mathsmaller{a_+ x_1 b_+}\\[-4pt] \mathsmaller{x_1 c_+ x_2}\end{array}& \hspace{-4mm}
	{\color{blue}(3,1)\atop \bullet} \,\begin{array}{l}\mathsmaller{a_+ c_+ x_2} \\[-4pt] \mathsmaller{x_1 c_+ b_+}\end{array}& \hspace{-4mm} \,
	{\color{blue}(4,0)\atop \bullet} \,\,\begin{array}{l}\mathsmaller{a_+ c_+ b_+}\\[-4pt] \mathsmaller{c_+^2 x_2}\end{array} &  \hspace{-4mm}
	{\color{blue}(5,-1)\atop \bullet}\,\mathsmaller{c_+^2 b_+}
	\\[10pt]
	{\color{blue}(-2,4)\atop \bullet}\,\mathsmaller{x_1^2 c_-} &  \hspace{-4mm}
	{\color{blue}(-1,3)\atop \bullet}\begin{array}{l}\mathsmaller{a_- x_1 x_2}\\[-4pt] \mathsmaller{x_1^2 b_-} \end{array} &  \hspace{-4mm}
	{\color{blue}(0,2)\,\,\atop \bullet} \begin{array}{l}\mathsmaller{a_+ x_1 b_-} \\[-4pt] \mathsmaller{x_1^2 x_3} \\[-4pt] 
	\mathsmaller{x_1 x_2^2}\\[-4pt]
    \end{array} &  \hspace{-4mm}
	{\color{blue}(1,1)\atop \bullet} \,\,\begin{array}{l}\mathsmaller{a_+ x_1 x_3}\\[-4pt] \mathsmaller{a_+ x_2^2}\\[-4pt]
		\mathsmaller{x_1 x_2 b_+}\end{array}& \hspace{-4mm}
	{\color{blue}(2,0)\atop \bullet} \,\begin{array}{l}\mathsmaller{a_+ x_2 b_+} \\[-4pt] \mathsmaller{x_1 c_+ x_3}\\[-4pt] 
		\mathsmaller{c_+ x_2^2}\end{array}&  \hspace{-4mm}
	{\color{blue}(3,-1)\atop \bullet} \begin{array}{l}\mathsmaller{a_+ c_+ x_3}\\[-4pt] \mathsmaller{c_+ x_2 b_+}\end{array} &  \hspace{-4mm}
	{\color{blue}(4,-2)\atop \bullet}\,\mathsmaller{c_+^2 x_3}
	\\[4pt]
	{\color{blue}(-3,3)\atop \bullet}\,\mathsmaller{a_- x_1 c_-}\, &  \hspace{-4mm}
	{\color{blue}(-2,2)\atop \bullet} \begin{array}{l}\mathsmaller{a_- x_1 b_-}\\[-4pt] \mathsmaller{x_1 c_- x_2} \end{array} & \hspace{-5mm} 
	{\color{blue}(-1,1)\atop \bullet} \begin{array}{l}\mathsmaller{a_- x_1 x_3} \\[-4pt] \mathsmaller{a_- x_2^2 } \\[-4pt] 
	\mathsmaller{x_1 x_2 b_-} \\[-4pt]
    \end{array}&  \hspace{-4mm}
	{\color{blue}(0,0)\atop \bullet} \,\,\begin{array}{l}\mathsmaller{a_+ x_2 b_-}\\[-4pt] \mathsmaller{x_1 x_2 x_3}\\[-4pt]
	\mathsmaller{ x_2^3 }\\[-4pt]\mathsmaller{a_+ a_- x_3}\\[-4pt]
    \end{array}& \hspace{-5mm}
	{\color{blue}(1,-1)\atop \bullet} \begin{array}{l}\mathsmaller{a_+ x_2 x_3} \\[-4pt] \mathsmaller{x_1 b_+ x_3 }\\[-4pt] 
	\mathsmaller{x_2^2 b_+}\\[-4pt] 
    \end{array}&  \hspace{-4mm}
	{\color{blue}(2,-2)\atop \bullet} \begin{array}{l}\mathsmaller{a_+ b_+ x_3 }\\[-4pt] \mathsmaller{c_+ x_2 x_3}\end{array} &  \hspace{-4mm}
	{\color{blue}(3,-3)\atop \bullet}\,\mathsmaller{c_+ b_+ x_3}
	\\[17pt]
	{\color{blue}(-4,2)\atop \bullet}\,\mathsmaller{x_1 c_-^2} &  \hspace{-4mm}
	{\color{blue}(-3,1)\atop \bullet} \begin{array}{l}\mathsmaller{a_- c_- x_2}\\[-4pt] \mathsmaller{x_1 c_- b_-} \end{array} &  \hspace{-5mm}
	{\color{blue}(-2,0)\atop \bullet} \begin{array}{l}\mathsmaller{a_- x_2 b_-} \\[-4pt] \mathsmaller{x_1 c_- x_3} \\[-4pt] 
		\mathsmaller{c_- x_2^2}\end{array} &  \hspace{-6.5mm}
	{\color{blue}(-1,-1)\atop \bullet} \begin{array}{l}\mathsmaller{a_- x_2 x_3}\\[-4pt] \mathsmaller{ x_1 b_- x_3 }\\[-4pt]
		\mathsmaller{x_2^2 b_-}\end{array}& \hspace{-5mm}
	{\color{blue}(0,-2)\atop \bullet} \begin{array}{l}\mathsmaller{a_+ b_- x_3 } \\[-4pt] \mathsmaller{x_1 x_3^2}\\[-4pt] 
	\mathsmaller{x_2^2 x_3}\\[-4pt] 
    \end{array}&  \hspace{-4mm}
	{\color{blue}(1,-3)\atop \bullet} \begin{array}{l}\mathsmaller{x_2 b_+ x_3 }\\[-4pt] \mathsmaller{a_+ x_3^2}\end{array} &  \hspace{-4mm}
	{\color{blue}(2,-4)\atop \bullet}\,\mathsmaller{c_+ x_3^2}
	\\[14pt]
	{\color{blue}(-5,1)\atop \bullet}\,\mathsmaller{a_- c_-^2} &  \hspace{-4mm}
	{\color{blue}(-4,0)\atop \bullet} \begin{array}{l}\mathsmaller{a_- c_- b_-}\,\\[-4pt] \mathsmaller{c_-^2 x_2} \end{array} &  \hspace{-6mm}
	{\color{blue}(-3,-1)\atop \bullet} \!\begin{array}{l}\mathsmaller{a_- c_- x_3} \\[-4pt] \mathsmaller{c_- x_2 b_-} \end{array}\, &  \hspace{-6mm}
	{\color{blue}(-2,-2)\atop \bullet} \begin{array}{l}\mathsmaller{a_- b_- x_3 }\\[-4pt] \mathsmaller{ c_- x_2 x_3 }\end{array}\,& \hspace{-6mm}
	{\color{blue}(-1,-3)\atop \bullet} \!\!\begin{array}{l}\mathsmaller{x_2 b_- x_3 } \\[-4pt] \mathsmaller{a_- x_3^2}\end{array}& \hspace{-4mm} 
	{\color{blue}(0,-4)\atop \bullet} \begin{array}{l}\mathsmaller{x_2 x_3^2}\\[-4pt] \mathsmaller{  b_+ b_-x_3}\end{array} & \hspace{-4mm} 
	{\color{blue}(1,-5)\atop \bullet}\,\mathsmaller{b_+ x_3^2 }
	\\[20pt]
	{\color{blue}(-6,0)\atop \bullet}\, \mathsmaller{c_-^3} & \hspace{-5mm}
	{\color{blue}(-5,-1)\atop \bullet}\,\mathsmaller{ c_-^2 b_-} & \hspace{-6mm}
	{\color{blue}(-4,-2)\atop \bullet}\,\mathsmaller{c_-^2 x_3} & \hspace{-6mm}
	{\color{blue}(-3,-3)\atop \bullet}\,\mathsmaller{c_- b_- x_3}&\hspace{-6mm}
	{\color{blue}(-2,-4)\atop \bullet}\,\mathsmaller{c _- x_3^2} & \hspace{-5mm}
	{\color{blue}(-1,-5)\atop \bullet}\,\mathsmaller{b_- x_3^2 } & \hspace{-4mm} 
	{\color{blue}(0,-6)\atop \bullet}\,\mathsmaller{x_3^3}
\end{array}
$$
Here vertices of the lattice and their bi-gradings are colored blue. 
\smallskip

The figure above suggests a conjecture that when placed on the ${\Bbb Z}^2$-grading  lattice the set of the degree $n$ irreducible monomials looks like a picture of the nested squares. The number of monomials corresponding to a given lattice point is constant on the border points of each square. It equals  1 on the border of the largest square, and it grows by 1 when one moves from the outer square to its nearest inner square. Assuming this is true we can calculate the number $d_n$ of independent irreducible monomials of degree $n$ in ${\cal F}_{q,v}$:
\be
\lb{D}
d_n\, =\, (n+1)\,+\, 8\,{n(n+1)(n+2)\over 6}\, =\, {2n+3\choose 3}, \quad n\geq 2.
\ee
This formula does not work only for $n=1$,  in which case $d_1=9 = {5\choose 3}-1$. 
So the Hilbert–Poincaré series for the algebra ${\cal F}_{q,v}(R)$ is
\be
\lb{HP-F3}
{\rm P}(t)\, =\, 1-t+ \sum_{n\geq 1} {2n+3\choose 3} t^n\, =\,\frac{8 t}{(1-t)^4} + \frac{1}{(1-t)^2} - t\, .
\ee

Let us summarize the results obtained.
\begin{prop}
\lb{prop2}
Any polynomial of degree $n\geq 0$ in the algebra ${\cal F}(R)_{q,v}$ can be expressed in terms of the $d_n={2n+3\choose 3}$ irreducible ${\Bbb Z}^2$-graded monomials. Below are listed the monomials for which both components of the ${\Bbb Z}^2$-grading are non-negative.
\ba
\nn
\mbox{\rm grad}= (0,2p), \, \mbox{$\mathsmaller{p\geq 0}$}: 
&&\hspace{-4mm} 
x_1^{i+p} x_2^{n-2i-p} x_3^i, \qquad\qquad\,\quad 0\leq i\leq 
\lfloor \textstyle{n-p\over 2}\rfloor,
\\[1pt]
\lb{0,2p}
&& \hspace{-4mm}
a_+ x_1^{i+p} x_2^{n-2i-p-2} \, b_- x_3^i,\,\,\,\,\quad  0\leq i \leq
\lfloor \textstyle{n-p-2\over 2}\rfloor,
\\
\nn
&& \hspace{-4mm}
a_+ a_- x_1^{n+p-3\over 2} x_3^{n-p-1\over 2}, \qquad\,\quad \mbox{if $(n-p)$ is odd}\, ,
\\[8pt]
\nn
\mbox{\rm grad}= (2k,2p), \, \mbox{$\mathsmaller{k>0,\,p\geq 0}$}: 
&& \hspace{-4mm}
x_1^{i+p} c_+^k x_2^{n-2i-p-k} x_3^i, \qquad\qquad\,\,0\leq i\leq 
\lfloor \textstyle{n-p-k\over 2}\rfloor,
\\[1pt]
\lb{2k,2p}
&& \hspace{-4mm}
a_+ x_1^{i+p} c_+^{k-1}x_2^{n-2i-p-k-1} b_+ x_3^i,\,\, 0\leq i \leq
\lfloor \textstyle{n-p-k-1\over 2}\rfloor,
\\[8pt]
\nn
\mbox{\rm grad}= (2k+1,2p+1), \, \mbox{$\mathsmaller{k\geq 0,\,p\geq 0}$}: 
&& \hspace{-4mm}
a_+ x_1^{i+p} c_+^k x_2^{n-2i-p-k-1} x_3^i, \quad\,\, 0\leq i\leq 
\lfloor \textstyle{n-p-k-1\over 2}\rfloor,
\\[1pt]
\lb{2k+1,2p+1}
&& \hspace{-4mm}
x_1^{i+p+1} c_+^k x_2^{n-2i-p-k-2}\, b_+ x_3^i, \,\, 0\leq i\leq 
\lfloor \textstyle{n-p-k-2\over 2}\rfloor .
\ea
The remaining monomials contain negative components in their ${\Bbb Z}^2$-grading. They are obtained from the monomials (\ref{0,2p})--(\ref{2k+1,2p+1}) via the (anti-)homomorphisms (\ref{antihom}), (\ref{hom}).
\end{prop}

\noindent {\bf Proof.}~
The expressions (\ref{0,2p})--(\ref{2k+1,2p+1}) for irreducible ${\Bbb Z}^2$-- graded degree $n$ monomials are extracted directly from the list of all irreducible monomials (\ref{list1})--(\ref{list6b}). Now, it is easy to count the number of irreducible ${\Bbb Z}^2$--graded monomials of degree $n$: $(n-p-k+1)$ pieces of degree $(2k,2p)$, and $(n-p-k)$ pieces of degree $(2k+1,2p+1)$. Thus, our calculation (\ref{D}) of the number of irreducible monomials of degree $n$ is correct.
\hfill$\blacksquare$

\subsection{Transition to the quotient algebra ${\cal F}_q(O(3))$ of quantized functions on the Lie group $O(3)$}

Keeping in mind that the contraction  $g$ has simple permutation relations in the algebra ${\cal F}(R)_{q,v}$ it seems reasonable to use it as one of the algebra generators. Indeed, one can use $g$ instead of $x_2^2$:
$$
x_2^2\, =\, q^2\, g \, +\, {\lambda -2\over q}\, a_+ b_-\,+\, {\lambda\over q}\, x_1 x_3 \, .
$$
With this substitution the ordering relations (\ref{rem-ord-6}), (\ref{rem-ord-8}), (\ref{rem-ord-9}), the orthogonality relation (\ref{ort-5}) and the diamond relations (\ref{diam-1}), (\ref{diam-4}) take a form
\ba
\lb{rem-ord-6-2}
x_3 x_1 &=&  {1+\lambda\over q^2} x_1 x_3\, +\, {\lambda\over q^2} a_+ b_-\, + \lambda q^2 (q-1) g\, ,
\\[2pt]
\lb{rem-ord-8-2}
b_+ a_- &=& {v^4\over q} b_- a_+\, =\, v^2\Bigl(  {1-\lambda\over q^2} a_+ b_-\, -\, {\lambda\over q^2} x_1 x_3 \, +\, \lambda q^2 g\Bigr)\, ,
\\[2pt]
\lb{ort-5-2}
c_+ c_- &=&  v^2\bigl( - q^{-1} a_+ b_-\, -\, q^{-1} x_1 x_3 \, +\, q^3 g\bigr)\, ,
\\[2pt]
\lb{diam-1-2}
a_+ a_- x_2 &=& -\, {1+2q\over q^2} a_+ x_1 b_- \, +\, v{1+q\over q^3}x_1\bigl( q^4 g\, -\,  x_1 x_3 \bigr)\, ,
\\[2pt]
\lb{diam-4-2}
x_2\, b_+ b_- &=& -\, {1+2q\over q^2} a_+ x_3 b_- \, +\, v{1+q\over q^3}\bigl( q^4 g\, -\,  x_1 x_3 \bigr)x_3\, ,
\ea

Then, all even powers of $x_2$ can be expressed in terms of $g$.
\begin{prop}
	\lb{prop-x2g}	
	\be	
	\lb{x2n}
	x_2^{2n} \, =\, \sum_{i\geq 0} F_{n,i}\,  g^{n-i} x_1^i x_3^i\, +\,
	v^{-2n}\sum_{i\geq 0} G_{n,i}\,  a_+\, g^{n-i-1} x_1^i x_3^i\, b_-\,  \quad\forall\, n>0\, ,
	\ee
	where $F_{n,i}$ and $G_{n,i}$ are some polynomials in $q^{\pm 1}$.
\end{prop}	
The proof of this statement follows the same pattern as the proofs of lemmas \ref{x1^n} and \ref{x2^n}; therefore, we shall provide here only two essential remarks regarding it.
 
 {1).}~ Formula (\ref{x2n}) is proved by induction on $n$. In the course of the proof, it will also be necessary to derive an expression for $(a_+ b_-)^n$. This expression contains the same terms as those in (\ref{x2n}), although with the different $q$-dependent coefficients.
		
 {2).}~ It is easy to find $F_{n,0}= q^{2n}$. We do not specify the explicit form of the coefficients $F_{n,i}$ and $G_{n,i}$.
 None of them is divisible by $\lambda$.
\phantom{a}\hfill$\blacktriangleright$\medskip

We conclude that in the list (\ref{list1})--(\ref{list6b}) of  irreducible monomials, the factors $x_2^{2k}$ and $x_2^{2k+1}$ can be replaced by $g^k$ and $g^k x_2$, respectively.
\smallskip

Now, let us assume that the contraction $g$ and the 3-elementary sum ${\goth e}_3$ are both central. From eqs. (\ref{G}), (\ref{O}), and (\ref{T-g-e3}), we conclude, that this is possible if and only if $v=1$. Factorizing the algebra ${\cal F}_{q,v=1}(R)$  by the relation 
$q^{2(n-1)}g=1$ (note that the element $q^{2(n-1)}g$ is group-like) we obtain the quotient algebra ${\cal F}_q(O(3))$. This is the standard quantization of the algebra of functions on the orthogonal Lie group $O(3)$ (cf.  \cite{FRT}). To conclude this section, we compute the Hilbert–Poincaré series for this algebra.
\smallskip

The list of irreducible monomials in ${\cal F}_q(O(3))$ is the subset of the list
(\ref{list1})--(\ref{list6b}), where the generator $x_2$ may appear in the powers $0$ and $1$ only.
Such restriction  excludes 
\begin{itemize}
	\item[*)] 
	$1+2^2+3^2+\dots + (n-1)^2={n(n-1)(2n-1)\over 6}$ elements of the total degree $n$ from the set (\ref{list1}) (the summands count monomials containing $x_2^n, x_2^{n-1}, \dots , x_2^2$).
	\item[**)] 4 times $1+3+6+\dots +{(n-1)(n-2)\over 2}={n(n-1)(n-2)\over 6}$ elements of the total degree $n$ from the sets (\ref{list2}) and (\ref{list3}) (the summands count monomials containing $X x_2^{n-1}, X x_2^{n-2}, \dots , X x_2^2$, where 
	$X$ is $a_+$, $a_-$, $b_+$, or $b_-$).
	\item[***)] 
	2 times $1+3+\dots +{(n-2)(n-3)\over 2}={(n-1)(n-2)(n-3)\over 6}$ elements of the total degree $n$ from the set (\ref{list4})  (the summands count monomials containing $X x_2^{n-2} Y, \dots ,  X x_2^2 Y$, where 
	$\{X,Y\}$ is either $\{a_+, b_+\}$, or $\{a_-, b_-\}$).
	\item[****)] 
	$1+2+3+\dots +(n-3)={(n-2)(n-3)\over 2}$ elements of the total degree $n$ from the set (\ref{list5})  (the summands count monomials containing $a_+ x_2^{n-2} b_-, \dots ,  a_+ x_2^2 b_-$).
\end{itemize}
Altogether we exclude from the list of irreducible monomials
$$
\textstyle
{n(n-1)(2n-1)\over 6}\, +\, 4\times {n(n-1)(n-2)\over 6}\, +\, 2\times {(n-1)(n-2)(n-3)\over 6}\, +\, 
{(n-2)(n-3)\over 2} \, =\, 8{n\choose 3}+1
$$
elements containing $x_2$  in  powers greater than 1.
The remaining sub-list contains
$$
\textstyle
{\tilde d}_0=1, \quad {\tilde d}_n = {2n+3\choose 3}-8 {n\choose 3} -1 = n(8n+1) \quad \forall n>0
$$
elements. This is the number of the irreducible degree $n$ monomials in the algebra $O_q(3)$.
So the Hilbert–Poincaré series for ${\cal F}_q(O(3))$ is
\be
\lb{HP-O3}
{\rm \tilde P}(t)\, =\, 1+ \sum_{n\geq 1} n(8n+1) t^n\, =\, 1+\frac{ t(7t+9)}{(1-t)^3} \, .
\ee

\subsection{Verification of the general theory results}
In this subsection, we verify the results presented in Section \ref{RTT-alg} regarding the general structure of QM-algebras, using the RTT-algebra ${\cal F}(R)_{q,v}$ as an example. The calculations utilize the linear basis constructed in the Section \ref{4.1}. As we have seen, the calculations prove to be quite cumbersome; the Mathematica symbolic computation software was used to carry them out.

The expansions of the generators of the characteristic subalgebra of the algebra ${\cal F}(R)_{q,v}$ in terms of ordered monomials have the following form.
\ba
\lb{g}
g &:=&  { q\over (q^2+q+1)}\,\Tr{1,2} \bigl( T_{1}T_{2}\, K_R\bigr)\, =\,
{2-\lambda\over q^3}\, a_+ b_- \, -\, {\lambda\over q^3}\, x_1 x_3\, +\, {1\over q^2}\, x_2^2,
\\[4pt]
\lb{e1}
\goth{e}_1&:=&   \tr_{\!R}\, T = {q^{-3}}\, x_1\, +\,{q^{-2}}\, x_2\, +\, {q^{-1}}\, x_3,
\\[4pt]
\nn
\goth{e}_2&:=&  \Tr{1,2}\bigl(T_{ 1} T_{ 2} A_R^{(2)}
\bigr)\, =\, 
{1\over q^5 v}\left(v x_1 x_2 - q a_+ a_-\right)\ 
\\
\lb{e2}
&&+\, 
{1\over q^4}\left((1+q)x_1 x_3 -q x_2^2+(\lambda -1)a_+ b_-\right)\,
+\, {1\over q^3 v}\left(v x_2 x_3- q b_+ b_-\right),
\\[4pt]
\nn
\goth{e}_3 &:=&  \Tr{1,2,3}\bigl(T_{ 1} T_{ 2} T_{ 3} A_R^{(3)}
\bigr)\, =\, {1+q^3\over q^6}\, x_1x_2x_3\, -\, {1\over q^3}\, x_2^3
\\
\lb{e3}
&&  -\,{2\over q^6 v}\left(1+q+\lambda q^2(q-1)\right)a_+a_-x_3\,-\,
{1\over q^6 v}\left(1+q^2+q^4-3\lambda q^2\right)a_+x_2b_-\, .
\ea
Here elements $g$ and $\goth{e}_3$ have ${\Bbb Z}^2$-grading $(0,0)$, while $\goth{e}_1$ and $\goth{e}_2$
have summands of the gradings $(0,2)$, $(0,0)$ and $(0,-2)$. They are placed in this order in their expressions (\ref{e1}) and (\ref{e2}). 
\smallskip

First, we have checked explicitly commutation relations (\ref{T-g-e3}).
Note that these computations involve the ordering and the orthogonality relations, and all the diamond relations too.
\smallskip

Relations (\ref{T-g-e3}) guarantee the commutativity of $g$ and $\goth{e}_3$ between themselves and with $\goth{e}_1$ and $\goth{e}_2$.  We have checked the commutativity of the elements  $\goth{e}_1$ and $\goth{e}_2$  too. For this check it is suitable to calculate the ${\Bbb Z}^2$-graded components $(0,4)$, $(0,2)$, \dots , $(0,-4)$ of the commutator $[\goth{e}_1,\goth{e}_2]$ separately.
\smallskip

Next, we checked the reciprocal relations
\be
\lb{reciprocal}
g \goth{e}_2 = \goth{e}_1 \goth{e}_3, \qquad g^2 \goth{e}_1 = \goth{e}_2 \goth{e}_3, \qquad g^3 = \goth{e}_3^2.
\ee
These calculations are quite labor-intensive, but they can be automated. We performed them using the Wolfram Mathematica software equipped with the NCAlgebra package, designed for operations with noncommutative associative algebras.
To illustrate the complexity of the calculations, we present the irreducible expressions for $g^2$ and $g^3$
\ba
\nn
g^2 &=& \mbox{$\mathsmaller{   {1\over q^4}\, x_2^4\, +\, 
		{(1 + q) (3 + 9 q + 5 q^2 - 12 q^3 - 8 q^4 + 11 q^5 + 5 q^6 - 6 q^7 + q^8)\over q^{10}}\,x_1 x_2^2 x_3\, -\,
		{(1 + q)^2 (3 + 6 q - 2 q^2 - 9 q^3 + 12 q^5 - 7 q^6 + q^7)\over q^{10}}\, x_1^2 x_3^2} $}
\\
\lb{g2}
&&
\mbox{$\mathsmaller{ +\, {(1 + 2 q - q^2) (1 + 2 q - 2 q^3 + q^4)\over q^8 v^2} \, a_+ x_2^2 b_-\, +\,
		{(1 + q) (6 + 15 q - 3 q^2 - 27 q^3 + 2 q^4 + 26 q^5 - q^6 - 17 q^7 + 8 q^8 - q^9)\over q^9 v^2}\, a_+ x_1 x_3 b_-
	} $}
\\[8pt]
\nn
g^3&=& \hspace{-2mm}\mbox{$\mathsmaller{{1\over q^6}\,x_2^6\, +\,
		{(1 + q) (5 + 30 q + 61 q^2 - 162 q^4 - 125 q^5 + 
			210 q^6 + 245 q^7 - 177 q^8 - 228 q^9 + 132 q^{10} + 110 q^{11} - 
			82 q^{12} - 14 q^{13} + 27 q^{14} - 9 q^{15} + q^{16})\over q^{16}}\, x_1 x_2^4 x_3 } $}
\\
\nn
&&
\hspace{-2mm}\mbox{$\mathsmaller{
		+\,{ (1 + q)^2\over q^{21}}\, (40 + 240 q + 405 q^2 - 444 q^3 - 2150 q^4 - 
		943 q^5 + 4418 q^6 + 4695 q^7 - 5112 q^8 - 8571 q^9 + 3935 q^{10} + 
		9597 q^{11} - 2640 q^{12} 
	} $}
\\[-2pt]
\nn		
&&\hspace{22mm}\mbox{$\mathsmaller{				
		- 7377 q^{13} + 2130 q^{14} + 3863 q^{15} - 
		1692 q^{16} - 1144 q^{17} + 877 q^{18} + 31 q^{19} - 210 q^{20} + 86 q^{21} - 
		15 q^{22} + q^{23}) \, x_1^2 x_2^2 x_3^2
	} $}
\\[4pt]
\nn
&&
\hspace{-2mm}\mbox{$\mathsmaller{
		-\, {(1 + q)^3\over  q^{21}}\, (40 + 200 q + 205 q^2 - 649 q^3 - 1501 q^4 + 563 q^5 + 3880 q^6 + 
		851 q^7 - 5999 q^8 - 2699 q^9 + 6637 q^{10} + 3166 q^{11} 
	} $}
\\[-2pt]
\nn		
&&\hspace{22mm}\mbox{$\mathsmaller{				
		- 5766 q^{12} - 1829 q^{13} +3950 q^{14} + 54 q^{15} - 1777 q^{16} + 582 q^{17} + 332 q^{18} - 311 q^{19} + 
		102 q^{20} - 16 q^{21} + q^{22})\, x_1^3  x_3^3
	} $}	
\\[4pt]
\nn
&&
\hspace{-2mm}\mbox{$\mathsmaller{
		+\, {(1 + 2 q - q^2) (1 + 3 q + q^2 - 3 q^3 + q^4) (1 + q - q^2 - q^3 + 
			q^4)\over q^{12} v^4 }\, a_+ x_2^4 b_-
	} $}
\\[4pt]
\nn
&&
\hspace{-2mm}\mbox{$\mathsmaller{
		+\,{(1 + q)\over q^{17} v^{4}}\, (20 + 125 q + 245 q^2 - 79 q^3 - 918 q^4 - 703 q^5 + 
		1485 q^6 + 2098 q^7 - 1420 q^8 - 3049 q^9 + 1064 q^{10} + 
		2845 q^{11} 
	} $}
\\[-2pt]
\nn		
&&\hspace{23mm}\mbox{$\mathsmaller{				
		- 914 q^{12} - 1779 q^{13} +841 q^{14} + 622 q^{15} - 517 q^{16} - 17 q^{17} + 145 q^{18} - 66 q^{19} + 
		13 q^{20} - q^{21}) \, a_+ x_1  x_2^2 x_3 b_- 
	} $}		
\\[4pt]
\nn
&&
\hspace{-2mm}\mbox{$\mathsmaller{
		+\,{(1 + q)^2\over q^{18} v^{4}}\, (80 + 420 q + 465 q^2 - 1413 q^3 - 3542 q^4 + 1082 q^5 + 
		9598 q^6 + 2984 q^7 - 15431 q^8 - 9126 q^9 + 17532 q^{10}
	} $}
\\[-2pt]
\nn		
&&\hspace{17mm}\mbox{$\mathsmaller{				
		+ 12266 q^{11} - 15685 q^{12} -9866 q^{13} + 11647 q^{14} + 4535 q^{15} - 6914 q^{16} - 400 q^{17} + 
		2738 q^{18} - 773 q^{19}
	} $}
\\[-2pt]
\lb{g3}		
&&\hspace{75mm}\mbox{$\mathsmaller{		
		- 455 q^{20} + 382 q^{21} - 116 q^{22} + 17 q^{23} - 
		q^{24}) \, a_+ x_1^2   x_3^2 b_- 
	} $}				
\ea

Our final check concerns the characteristic identity (\ref{orthoCH-odd}) for the quantum matrix $T$. In the case of the $O(3)$-type RTT algebra, this identity looks like
\be
\lb{CH-O3}
T^{\overline{3}} g\, -\, T^{\overline{2}}\left( q\, g\, {\goth e}_1 +(1-q) {\goth e}_3\right)\, +\,
T\left( q\, g\, {\goth e}_2 +q(q-1)\, g^2)\right)
\, -\, {\rm Id}_V\, q^2 g\, {\goth e}_3\, =\, 0.
\ee
To verify this, an explicit formula is required for the linear map $\phi_R$ (\ref{phi}) involved in the construction of the powers (\ref{Tk}) of the quantum matrix $T$. The map is given by

\be
\phi_R(T)\, =\, \left(
\begin{array}{ccc}
{1\over q^2} x_1 + {\lambda\over q^2} x_2+ {\lambda (q-1)\over q^2} x_3 & {v\over q^3} a_+ -{\lambda\over q^2} b_+
& {v^2\over q^4} c_+
\\[5pt]
{1\over q^2 v} a_- -{\lambda\over q^2} b_- & {1\over q^2} x_2 +{\lambda\over q} x_3 & {v\over q^2} b_+
\\[5pt]
{1\over q^2 v^2} c_- & {1\over q v} b_- & x_3
\end{array}
\right) .
\ee
Direct calculations with the use of Wolfram Mathematica confirm the identity (\ref{CH-O3}).

\appendix
\section{Multiparametric orthogonal R-matrices}

\lb{primcontr}

\def\theequation{\thesection.\arabic{equation}}
\makeatletter\@addtoreset{equation}{section}\makeatother

Here we describe multiparametric $O(N)$ type R-matrices \cite{Sch,I} (see also 
\cite{IOP3}), which can be derived from the standard orthogonal R-matrices \cite{FRT} by 
the twist procedure \cite{Resh}. We give a parameterization for these R-matrices in terms 
of independent parameters.
\smallskip

Let $E_{ij}$ be $N\times N$  matrix units. Denote $i':=N+1-i$,~ $n:=\lfloor N/2\rfloor$. Define 
$$
\rho_i=-\rho_{i'}:=\frac{N}{2}-i\,\,  \forall\,i=1,\dots, n,
$$ 
and  additionally $\rho_{n+1}=0$ for odd $N$ (i.e., if $N=2n+1$).
The multiparametric $O(N)$ type R-matrix reads   
\be
\lb{multiR}
R = \sum_{i,j=1}^{N} u_{ij}\, E_{ij}\otimes E_{ji} + (q-q^{-1})\!
\sum_{i,j=1\atop i<j}^N E_{ii}\otimes E_{jj} -(q-q^{-1})\!
\sum_{i,j=1\atop i<j'}^N v_{ij}\,q^{-(\rho_i+\rho_j)} E_{ij}\otimes E_{i'j'}
\ee
Here parameters $u_{ij}\in {\Bbb C}\setminus \{0\}$ are independent if $1\leq i<j\leq n$. For the other 
values of their indices  we describe $u_{ij}$ separately for the cases $N$  even and odd:
\ba
\lb{par-u-even}
N=2n:&& u_{ii}=q\,\, \forall\, i=1,\dots ,N,\qquad 
u_{ji}=u_{ij}^{-1}\,\, \forall\, i,j=1,\dots n\, :\, i\neq j, 
\\[1pt]
\nn
&&  u_{i'j}= (u_{ij} x_j)^{-1}, \quad u_{ij'}=x_i u_{ij}^{-1}, \quad
u_{i'j'}=u_{ij}x_j x_i^{-1}\,\, \forall i,j=1,\dots ,n,
\\[5pt]
\lb{par-u-odd}
N=2n+1:&& u_{ii}=q^{1-\delta_{i,n+1}}\,\, \forall\, i=1,\dots ,N,\qquad 
u_{ji}=u_{ij}^{-1}\,\, \forall\, i,j=1,\dots n\, :\, i\neq j, 
\\[1pt]
\nn
&&
u_{i,n+1} = u_{n+1, i}^{-1} = u_{n+1,i'}= u_{i',n+1}^{-1}= y_i\,\, \forall i=1,\dots ,n,
\\[1pt]
\nn
&&  u_{i'j}= (u_{ij} y^2_j)^{-1}, \quad u_{ij'}=y^2_i u_{ij}^{-1}, \quad
u_{i'j'}=u_{ij}y^2_j y_i^{-2}\,\, \forall i,j=1,\dots ,n,
\ea
where parameters $x_i, y_i \in {\Bbb C}\setminus\{0\}$, $i=1,\dots ,n$, are independent.

For parameters $v_{ij}$ in (\ref{multiR}) one has following expressions
\ba
\lb{par-v-even}
N=2n:&& 
v_{ij}=z_j z_i^{-1}, 
\quad  v_{i'j}= v_{ij} x_i^{-1}, \quad v_{ij'}=v_{ij} x_j\,\, \forall i,j=1,\dots ,n,
\\[5pt]
\lb{par-v-odd}
N=2n+1:&& v_{ij}=z_j z_i^{-1}\,\, \forall i,j=1,\dots ,n+1,
\\[1pt]
\nn
&&
v_{i'j}= v_{ij} y_i^{-2}, \quad v_{ij'}=v_{ij} y_j^2\,\, \forall i,j=1,\dots ,n,
\ea
where $z_i\in {\Bbb C}\setminus \{0\}$, and one can set $z_1=1$ without loss of generality.

Altogether the R-matrix (\ref{multiR}) contains $n(n+3)/2-1$ parameters in case 
$N=2n$,
and $n(n+3)/2$ parameters in case $N=2n+1$.

\section*{Acknowledgments}

The authors thank Dimitry Gurevich and Pavel Saponov for their time, fruitful 
discussions and valuable remarks.
The work of the second author (P. P.) was  carried out within
the project "Symmetry. Information. Chaos" (HSE-BR-2025-84) of the Basic Research Program at HSE University.

\end{document}